\documentclass{article}
\usepackage{amsthm}
\usepackage{amsmath}
\usepackage{amssymb}
\usepackage{physics}
\newtheorem{theorem}{Theorem}
\newtheorem{corollary}{Corollary}
\newtheorem{proposition}{Proposition}

\theoremstyle{definition}
\newtheorem{example}{Example}

\usepackage[ruled,titlenotnumbered,vlined]{algorithm2e}
\SetKwInput{Input}{Input}
\SetKwInput{Output}{Output}
\IncMargin{0.5em}
\SetNlSkip{0.5em}

\usepackage{comment}

\usepackage{hyperref}
\hypersetup{
    colorlinks=true,
    linkcolor=blue,
    filecolor=magenta,
    urlcolor=cyan,
    pdfpagemode=FullScreen,
    }

\usepackage[capitalize]{cleveref} 

\begin{document}

\title
{General formulas for a class of Euler sums}
\author{David H. Bailey,\\
Lawrence Berkeley National Laboratory, USA, \\
{\tt dhbailey@lbl.gov} \\
Ross McPhedran,\\
School of Physics, University of Sydney 2006, Australia,\\
{\tt ross.mcphedran@sydney.edu.au} \\
and Bruno Salvy, \\
Institut National de Recherche en Sciences et Technologies du Numérique (INRIA), \\
{\tt bruno.salvy@inria.fr}}
\maketitle

\begin{abstract}
Let $H_k = 1 + 1/2 + 1/3 + \cdots + 1/k$ denote the $k$th harmonic number. We present an easy-to-implement algorithm for the computation of explicit closed-form evaluations, in terms of the digamma and polygamma functions, for Euler sums of the form 
\begin{align}
\sum_{k=1}^\infty R(k) H_k,
\end{align}
where $R(k)$ is a rational function (quotient of two polynomials) whose denominator degree is at least two larger than the numerator degree. We apply the same method to show how the computation of a general formula for Euler sums of the form
\begin{align*}
\sum_{k=1}^\infty \frac{H_k}{(m_1 k + n_1)^{p_1} (m_2 k + n_2)^{p_2} \cdots (m_r k + n_r)^{p_r}}
\end{align*}
reduces to partial fraction decomposition. We present explicit formulae for sums with one or two terms in the denominator, with powers $p_i$ ranging up to 3, and with multipliers $m_i$ ranging up to 4. We also include results for related Euler sums such as 
\begin{align*}
\sum_{k=1}^\infty \frac{k^q H_k}{(m k + n)^p}.
\end{align*}
Computation of Euler sums directly to very high precision enables us to rigorously check the above-mentioned formulas in many specific cases. 
\end{abstract}

\section{Introduction}

An \emph{Euler sum} (also termed \emph{Euler-Zagier sum} or merely \emph{harmonic sum}) is an infinite series involving the harmonic numbers $H_k = 1 + 1/2 + 1/3 + \cdots + 1/k$. A few simple examples include $\sum_{k=1}^\infty H_k/k^2$, $\sum_{k=1}^\infty H_k^2/k^3$, $\sum_{k=1}^\infty H_k/(3k+2)^3$ and $\sum_{k=1}^\infty \sum_{j=1}^k \sum_{i=1}^j 1/(ijk)$. Such sums arise in mathematical physics, in the study of the Riemann hypothesis and in numerous other contexts. They have been studied since the time of Euler, and more recently in \cite{Ablinger2009,Ablinger2011,BBG1,BBG2,ChoiSri,FlajoletSalvy,SofoChoi,Zheng}. One notable feature of these sums is that many have surprisingly elegant analytic evaluations, for example:
\begin{align*}
\sum_{k=1}^\infty \frac{H_k}{k^{3}} &= \frac{ 5}{4} \zeta(4) = \frac{1}{72} \pi^4 \nonumber \\
\sum_{k=1}^\infty \frac{H_k}{(k+1)^{5}} &= \frac{ 1}{4}\left(  3\zeta(6) -2\zeta(3)^2\right) = \frac{1}{1260} \left( \pi^6 - 630 \zeta(3)^2\right) \nonumber \\
\sum_{k=1}^\infty \frac{H_k^3}{k^4} &= \frac{693}{48} \zeta (7)+2 \zeta (5) \zeta (2)-\frac{51}{4} \zeta (4) \zeta (3) \nonumber \\
\sum_{k=1}^\infty \frac{H_k}{(k+1)^{7}} &= \frac{ 1}{4}\left(  5\zeta(8) -4\zeta(3)\zeta(5)\right) = \frac{1}{4} \left(\frac{\pi^8}{1890} - 4 \zeta(3) \zeta(5)\right).
\end{align*}

In an earlier study \cite{McPBai2025}, we presented results for mixed-denominator Euler sums of the form
\begin{align}
\sum_{k=1}^\infty \frac{H_k^m}{k^{n_0} (k+1)^{n_1} (k+2)^{n_2} \cdots (k+t)^{n_t}},
\end{align}
for nonnegative integers $m$ and $(n_i)$, with $m \geq 1$ and $n_0 + n_1 + \cdots + n_t \geq 2$. We showed that these have closed-form evaluations in terms of the Riemann zeta function, at least when $m + n_0 + \cdots + n_t \leq 12$. We also presented a catalog of several hundred formulas, produced using high-precision numerical methods.

In this article, we illustrate a general approach, based on residues, to evaluate sums of the form
\[\sum_{k=1}^\infty R(k) H_k, \]
where $R(k)$ is any rational function with denominator degree at least two larger than the numerator degree, and no poles at integers except possibly at 0. This leads to explicit formulas and techniques for sums such as
\[\sum_{k=1}^\infty \frac{H_k}{(m_1 k + n_1)^{p_1} (m_2 k + n_2)^
{p_2} \cdots (m_r k + n_r)^{p_r}} \quad \text{and} \quad \sum_{k=1}^\infty \frac{k^q H_k}{(m_1 k + n_1)^{p_1}}\]
for arbitrary $(m_i,n_i,p_i, q)$, provided the sum is convergent. We present explicit formulae, expressed in terms of the digamma and polygamma functions, for sums with one or two terms in the denominator, with powers $p_i$ ranging up to 3, and with multipliers $m_i$ ranging up to 4.

\section{Formulas by residue computation}

\subsection{Background on the \texorpdfstring{$\psi$}{psi} function}
The $\psi$ function (also known as the \emph{digamma function}) is the logarithmic derivative of the gamma function:  $\psi(z) = \Gamma'(z)/\Gamma(z)$. Many formulas for $\psi$ follow by differentiation from the recurrence formula $\Gamma(s+1)=s\Gamma(s)$, the reflection formula~$\Gamma(s)\Gamma(1-s)=\pi/\sin(\pi s)$ and the duplication formula $\Gamma(2s)=\pi^{-1/2}2^{2s-1}\Gamma(s)\Gamma (s+1/2)$, as well as its generalization known as Gauss' multiplication formula, see \cite[ch.~5]{dlmf}. In particular, this last formula gives values of~$\psi(p/q)$ at rational points~\cite[5.4.19]{dlmf}.

The $\psi$ function is meromorphic (i.e., analytic except for a discrete set of poles), with simple poles at nonpositive integers. In a neighborhood of $s=0$, it possesses the expansion~\cite[5.5.2 and 5.7.4]{dlmf}
\begin{equation}\label{eq:psi_at_0}
\psi(-s)+\gamma=\frac1s-\sum_{k\ge2}\zeta(k)s^{k-1},
\end{equation}
where $\gamma = 0.5772156649\ldots$ is Euler's constant and $\zeta (k)=1+1/2^k+1/3^k+\dotsb$ is the Riemann zeta function. The $\psi$ function is closely related to the harmonic numbers: A direct consequence of the recurrence formula for~$\Gamma$ is that for $n$ a positive integer, $\psi(n+1)=H_n - \gamma$.

The polygamma function is merely the $n$-th derivative of the digamma function: $\psi^{(n)} (z)$, also denoted $\psi_n(z)$ or $\psi(n,z)$; note that $\psi(z) = \psi_0(z) = \psi(0,z)$. For positive integer $n$, the polygamma function can be written in terms of the Hurwitz zeta function as $\psi^{(n)}(z) = (-1)^{n+1} n! \zeta (n+1, z)$.

Some interesting special values of the digamma and polygamma functions include \cite{Choi2007}:
\begin{align}
\psi(1/2) &= - \gamma - 2 \log 2 \quad &  \psi(1/4) &= -\gamma - 3\log (2) - \pi/2  \nonumber \\
\psi^{(1)}(1) &= \pi^2 / 6 & \quad \psi^{(1)}(1/2) &= \pi^2/2 \nonumber \\
\psi^{(2)}(1) &= -2 \zeta(3) \quad & \psi^{(2)}(1/2) &= -14 \zeta(3) \nonumber \\
\psi^{(2)}(1/4) &= -2 \pi^3 - 56 \zeta(3) \quad & \psi^{(3)}(1/2) &= \pi^4
\end{align}

\subsection{Harmonic sums by residue computation}

An interesting consequence of the behavior of $\psi$ in the neighborhood of negative integers is the following:
\begin{proposition} 
\label{thm:lem1}
Let $R(z)$ be a rational function (quotient of two polynomials) that is $O(z^{-2})$ at infinity with no poles at integers except possibly at 0. Then
\begin{equation}
\sum_{k=1}^\infty R(k) H_k = 
\frac12\sum_{\alpha \text{ pole of } R(s)}
\Res\!\left[R'(s)
(\psi(-s)+\gamma)-R(s) (\psi(-s) + \gamma)^2\right]_{s=\alpha}, \label{form:BSlemma}
\end{equation}
where $\Res [Q(s)]_{s=\alpha}$ denotes the residue, in the complex analysis sense, of the function $Q(s)$ at $s = \alpha$.
\end{proposition}
\begin{proof}
This is a direct corollary of \cite[Eqs.~(2-6) and (2-9)]{FlajoletSalvy}:
\begin{align*}
\sum_{k=1}^\infty{r(k)}&=-\mathcal R[r(s)(\psi(-s)+\gamma)],\\
2\sum_{k=1}^\infty{r(k)H_k}+\sum_{k=1}^\infty{r'(k)}&=
-\mathcal R[r(s)(\psi(-s)+\gamma)^2],
\end{align*}
where $\mathcal R[r(s)\xi(s)]$ with $\xi\in\{\psi (-s)+\gamma,(\psi(-s)+\gamma)^2\}$ denotes the sum of the residues of $r(s)\xi(s)$ at the poles of~$r(s)$ and at~0.

Combining these two formulas gives the sum in the proposition. With this combination, the pole at~0 does not need to be considered separately: if $R(s)$ does not have a pole at~0, then from  \cref{eq:psi_at_0} and its derivative, we get 
\begin{align*}
-R(s)(\psi(0,-s)+\gamma)^2 &= -R(0)/s^2 - R'(0)/s + O(1),\\
R'(s)(\psi(0,-s)+\gamma) &=R'(0)/s+O(1).
\end{align*}
Thus the residue at~0, namely the coefficient of $1/s$, is~0.
\end{proof}

\begin{example}
We first illustrate the algorithm that follows from \cref{thm:lem1} for the rational function~$R(s)=1/(2s-1)^2$. Its only pole is at~$\alpha=1/2$. At this point, we use the expansion
\[\psi(-s)+\gamma=2(1-\ln 2)-\left(\frac{\pi^2}2+4\right)(s-1/2)+
(8-7\zeta(3))(s-1/2)^2+O((s-1/2)^3).\]
It follows that the residues of $-R(s)(\psi(-s)+\gamma)^2$ and $R'(s) (\psi(-s)+\gamma)$ at~$s=1/2$ are
\[(1-\ln 2)\left(\frac{\pi^2}2+4\right) \quad\text{and}\quad \frac72\zeta(3)-4, \]
which gives
\[\sum_{k=1}^\infty{\frac{H_k}{(2k-1)^2}}=\frac{\pi^2}4-2\ln 2-\frac{\pi^2\ln 2}4+\frac74\zeta(3).\]
\end{example}
\subsection{Algorithm}
When turning \cref{thm:lem1} into an algorithm suitable for implementation in a computer algebra system, it is important to observe that exact or approximate expressions of the poles of~$R(s)$ are not needed for the computation of the residues; they can be obtained in a purely symbolic way. Assume that $R(s)=P(s)/Q(s)$ with $P,Q$ polynomials that are relatively prime (this can ensured by reducing the fraction to a common denominator and dividing numerator and denominator by their gcd). If~$\alpha$ is a root of~$Q$ of multiplicity~$p\in\mathbb N$, then the Taylor expansion of~$Q(s)$ as~$s\rightarrow\alpha$ factors as
\[Q(s)=(s-\alpha)^p\left(\frac{Q^{(p)}(\alpha)}{p!} + \frac{Q^{(p+1)} (\alpha)} {(p+1)!}(s-\alpha)+\dotsb\right).\]
Then computing the Taylor expansions of $P(s)$, $\psi(-s)+\gamma$ and the reciprocal of the
second term in the product above, all at precision~$O((s-\alpha)^{p+1})$ is sufficient to recover the residues required by the lemma. 

From the computer algebra point of view, a natural strategy is to start from a \emph{square-free factorization}  of~$Q$, ie, a factorization of the form~$Q=Q_1Q_2^2\dotsm Q_r^r$ where the polynomials~$Q_i$ do not have multiple roots and are relatively prime. Such a factorization can be computed efficiently using only gcd computations~\cite[\S14.6]{GathenGerhard}. Then, for each~$p=1,\dots,r$, one computes a residue as described above for~$Q_p$ and sums the result symbolically over all roots of~$Q_p$. This approach completely avoids factorization and is usually the preferred one in computer algebra algorithms. 

However, for users whose rational functions~$R(s)$ have simple expressions in~$\mathbb Q(s)$, more explicit formulas are obtained by first factoring the denominator into irreducible factors in~$\mathbb Q[s]$ and then proceeding as above. This is the choice made in our implementation presented in Appendix~\ref{app:maple_code} and results in Algorithm~\ref{algo:algo1}.

\begin{algorithm}[t]
\DontPrintSemicolon
\SetAlgoRefName{HarmonicSum}
\Input{$R=P/Q$ a rational function\\ 
\qquad with $P,Q$ in $\mathbb Q[s]$,
$\gcd(P,Q)=1$, $\deg Q-\deg P\ge2$ and $0\not\in Q(\mathbb
N\setminus\{0\})$}
\Output{$\sum_{k\ge1}R(k)H_k$, where $H_k$ is the
$k$th harnomic number}
$q_1^{e_1}\dotsm q_m^{e_m}:=\operatorname{factor}(Q)$\tcp*{$e_i>0$,
$q_i\in\mathbb Q[s]$ irreducible}
\For(\tcp*[f]{loop over factors}){$i:=1$ \KwTo $m$}{
    $r:=Rq_i^{e_i}$ \tcp*{rest of the fraction}
    $q:=(q_i(s)-q_i(\alpha))/(s-\alpha)\in\mathbb Q(\alpha)[s]$
    \tcp*{$\alpha$ stands for the pole}
    $S:=\operatorname{Taylor}(r/q,s=\alpha,e_i+1)$
    \tcp*{Taylor expansion modulo $O((s-\alpha)^{e_i+1})$}
    $S:=\operatorname{Laurent}(S/(s-\alpha)^{e_i},s=\alpha,1)$
    \tcp*{Laurent expansion of $R$ modulo $O(s-\alpha)$}
    $S:=\operatorname{Laurent}((\psi(-s)+\gamma)S'-(\psi
    (-s)+\gamma)^2S,s=\alpha,0)$
    \tcp*{expansion modulo $O((s-\alpha)^0)$}
    $c:=\operatorname{coeff}(S,s-\alpha,-1)$
    \tcp*{residue from \cref{thm:lem1}}
    $v_i:=\frac12\sum_{\alpha\mid q_i(\alpha)=0}{c}$
    \tcp*{sum over conjugate roots of $q_i$}
}
\Return{$\sum_{i=1}^m{v_i}.$}
\tcp*{sum the contributions from all poles}
\caption{Sum harmonic number times rational function
\label{algo:algo1}}
\end{algorithm}

\begin{example}\label{ex:2km1square}
With our Maple implementation, the previous example is obtained as 
\begin{verbatim}
> expand(harmonicsum(1/(2*k-1)^2,k));
\end{verbatim}
\[{\color{blue}
\frac{7 \zeta \! \left(3\right)}{4}+\frac{\pi^{2}}{4}-2 \ln \! \left(2\right)-\frac{\ln \! \left(2\right) \pi^{2}}{4}
}
\]
\end{example}
\begin{example}\label{form:illcor3}
Here is an example with multiple poles
\[
\sum_{k=1}^\infty \frac{H_k}{(2k+1)^3 (3k+1)^2}. 
\]
Our implementation provides the following nice expression
\begin{verbatim}
> expand(harmonicsum(1/(2*k+1)^3/(3*k+1)^2,k));
\end{verbatim}
\vskip-1ex
{\color{blue}
\begin{multline*}
\frac{\pi^{4}}{16}
+42 \zeta \! \left(3\right)
+\frac{27 \pi^{2}}{2}
-7 \ln \! \left(2\right) \zeta \! \left(3\right)
-6 \ln \! \left(2\right) \pi^{2}
-108 \ln \! \left(2\right)^{2}\\
+27 \left(
    -\frac{\pi  \sqrt{3}}{6}-
    \frac{3 \ln \! \left(3\right)}{2}\right)^{2}
+\left(-27-\frac{\pi\sqrt{3}}{2}-\frac{9 \ln \! \left(3\right)}{2}\right)
    \psi'\! \left(\frac{1}{3}\right)
-\frac{3}2 \psi''\! \left(\frac{1}{3}\right)
\end{multline*}
}    
\end{example}
\begin{example}
An example that would be more involved by hand is
\[\sum_{k=1}^\infty{\frac{H_k}{(k^2+3k+1)^2}}\]
which is obtained as
\begin{verbatim}
> harmonicsum(1/(k^2+3*k+1)^2,k);
\end{verbatim}
\vskip-1ex
{\color{blue}
\begin{multline*}
-\frac{\sqrt{5}\, \psi \! \left(\frac{3}{2}+\frac{\sqrt{5}}{2}\right)^
{2}}{25}
-\frac{2 \sqrt{5}\, \gamma  \psi \! \left(\frac{3}{2}+\frac{\sqrt{5}}
{2}\right)}{25}
+\frac{\psi'\! \left(\frac{3}{2}+\frac{\sqrt{5}}
{2}\right) \psi \! \left(\frac{3}{2}+\frac{\sqrt{5}}{2}\right)}{5}\\
+\frac{\sqrt{5}\, \psi \! \left(-\frac{\sqrt{5}}{2}+\frac{3}
{2}\right)^{2}}{25}
+\frac{2 \sqrt{5}\, \gamma  \psi \! \left(-\frac{\sqrt{5}}{2}+\frac{3}
{2}\right)}{25}
+\frac{\psi'\! \left(-\frac{\sqrt{5}}{2}+\frac{3}
{2}\right) \psi \! \left(-\frac{\sqrt{5}}{2}+\frac{3}{2}\right)}{5}\\
-\frac{\psi''\! \left(-\frac{\sqrt{5}}{2}+\frac{3}
{2}\right)}{10}
-\frac{\psi''\! \left(\frac{3}{2}+\frac{\sqrt{5}}
{2}\right)}{10}\\
+\left(-\frac{\sqrt{5}}{25}+\frac{\gamma}{5}\right) \psi'\! \left(-
\frac{\sqrt{5}}{2}+\frac{3}{2}\right)+\left(\frac{\sqrt{5}}{25}+
\frac{\gamma}{5}\right) \psi'\! \left(\frac{3}{2}+\frac{\sqrt{5}}{2}\right)
\end{multline*}
}
\end{example}

\section{General formulas} \label{sec:theorems}
Algorithm \ref{algo:algo1} can deal with an arbitrary rational function~$R(s)\in\mathbb Q(s)$, provided the sum converges. It can easily be extended to parameterized rational functions in, say, $\mathbb Q(a_1,\dots,a_m)(s)$ for parameters~$a_1,\dots,a_m$. While it cannot handle poles whose order is a parameter, the method used by the algorithm can be ``run by hand'' and produce similar formulas in that case. But for relatively simple $R(s)$, useful general formulas can be derived, as we will show in the following.

\subsection{Simple poles}
The computation of the contribution of a simple pole of~$R$ to the sum in \cref{form:BSlemma} is straightforward. If $R(s)=P(s)/Q(s)$ with $P,Q$ polynomials that are relatively prime (this can be checked by a gcd computation), then as $s\rightarrow\alpha$,
\begin{gather*}
R(s)=\frac{c_\alpha}{s-\alpha}+O(1),\quad 
R'(s)=-\frac{c_\alpha}{(s-\alpha)^2}+O(1),\\
\psi(-s)+\gamma=\psi(-\alpha)+\gamma -\psi'(-\alpha)(s-\alpha)+O(
(s-\alpha)^2),
\end{gather*}
with $c_\alpha=P(\alpha)/Q'(\alpha)$. As a consequence,
\[
\Res[R'(s)(\psi(-s)+\gamma)-R(s)(\psi(-s)+\gamma)^2]_{s=\alpha}
=
c_\alpha(\psi'(-\alpha)-(\psi(-\alpha)+\gamma)^2).
\]
Using partial fraction decomposition then gives the following.
\begin{theorem} \label{thm:thm1}
For $t$ not an integer, define
\begin{align}
T(t,1) &= \frac{1}{2} \left(\psi'(t)-2\gamma\psi(t)-\psi(t)^2\right). \label{form:theorem4b}
\end{align}
Let $(t_1, t_2, \cdots, t_r)$, with $r \geq 2$, be distinct nonintegers, and let $(C_1, C_2, \cdots, C_r)$ be complex numbers that satisfy $C_1 + C_2 + \cdots + C_r = 0$ (necessary for convergence). Then
\begin{equation}
\sum_{k=1}^\infty H_k \left(\frac{C_1}{k+t_1} + \frac{C_2}{k+t_2} + \cdots + \frac{C_r}{k+t_r}\right) = \sum_{j=1}^r C_j T(t_j,1). \label{form:theoremx}
\end{equation}
\end{theorem}
\begin{proof}
The discussion above gives the result with $(\psi'(t)- (\psi(t)+\gamma)^2)/2$ in place of $T(t,1)$. The disappearance of the term $\gamma^2$ comes from the sum of the $C_j$ being~0.
\end{proof}

\begin{example}
In this example, the denominator has complex
roots:
\begin{align}
\sum_{k=1}^\infty \frac{H_k}{k^3 + 1}.
\end{align}
Let $-1$ and $c_{\pm} = (1 \pm i \sqrt{3})/2$ be the cubic roots of $-1$, so that $(k + 1) (k - c_+) (k - c_-) = k^3 + 1$. A partial fraction decomposition yields
\begin{align}
\frac{1}{k^3 + 1} &= \frac{1}{3 (k + 1)} - \frac{c_+}{3 (k - c_+)} - 
\frac{c_-}{3 (k - c_-)}.
\end{align}
Recall that $\psi(1) = -\gamma$ and $\psi'(1) = \pi^2/6$. Then
\begin{align*}
\sum_{k=1}^\infty \frac{H_k}{k^3 + 1} &= \frac{1}{3} T(1,1) - 
\frac{c_+}{3} T(-c_+,1) - \frac{c_-}{3} T(-c_-,1) \\
&=\frac16\left(\frac{\pi^2}6+\gamma^2
-\sum_{\epsilon\in\{+,-\}}c_{\epsilon}(\psi'(-c_\epsilon)-2\gamma\psi
(-c_\epsilon)-\psi (-c_\epsilon)^2)
\right)\\
&= 0.828902143400992508742\ldots.
\end{align*}
\end{example}

\subsection{A unique pole of higher order}
The next theorem is a variant of a formula that we initially discovered experimentally, using a combination of \emph{Wolfram Mathematica}, the Online Encyclopedia of Integer Sequences \cite{oeis} and high-precision numerical computing. 

\begin{theorem} \label{thm:thm2}
For $t$ not an integer, and integer $p \geq 1$, define
\begin{equation}
T(t,p)=
\frac{(-1)^{p-1}}{2(p-1)!}\left(\psi^{(p)}(t)
-2\gamma\psi^{(p-1)}(t)-\sum_
{k=0}^{p-1} \binom{p-1}{k}\psi^{(k)}(t)\psi^{(p-1-k)}(t)
\right)\label{eq:Ttp}.
\end{equation}
Then, for $p\ge2$,
\begin{equation}\label{form:theorem1}
\sum_{k=1}^\infty \frac{H_k}{(k + t)^p}=T(t,p).
\end{equation}
\end{theorem}
\noindent Note that the value of $T(t,p)$ in \cref{eq:Ttp} is
consistent
with the value of~$T (t,1)$ in \cref{thm:thm1}.

\begin{proof}
We apply \cref{thm:lem1} to $R(s) = 1 / (s + t)^p$, whose only pole
is at~$s=-t$. 

From $R'(s)=-p/(s+t)^{p+1}$ and the Taylor expansion at $s=-t$
\begin{align}
\psi(-s)+\gamma &= \psi(t) + \gamma + \sum_{k\ge1}{(-1)^k
\frac{\psi^{(k)}(t)}{k!}(s+t)^k} \label{form:res1}
\end{align}
it follows that
\[\operatorname{Res}[R'(s)(\psi(-s)+\gamma)]_{s=-t}=(-1)^
{p-1}\frac{\psi^{(p)}(t)}{(p-1)!}.\]
Next, squaring the expansion above and extracting the coefficient
of $(s+t)^{p-1}$ gives
\[\operatorname{Res}[R(s)(\psi(-s)+\gamma)^2]_{s=-t}=\frac{(-1)^{p-1}}{
(p-1)!}\left(
2\gamma\psi^{(p-1)}(t)+\sum_
{k=0}^{p-1} \binom{p-1}{k}\psi^{(k)}(t)\psi^{(p-1-k)}(t)\right)
.\]
Summing both contributions gives the result.
\end{proof}

Note that the summation in the last line of \cref{form:BSlemma} is symmetric. Thus it suffices to sum it only halfway, doubling each summand except possibly the middle one.

One direct consequence of Theorem \ref{thm:thm2} is the following:
\begin{corollary} \label{cor:cor1}
For integers $m, n \geq 1, \gcd (m, n) = 1, p \geq 2$,
\begin{equation}\label{form:theorem3}
\sum_{k=1}^\infty \frac{H_k}{(mk + n)^p}=
\frac{(-1)^{p-1}}{2m^p(p-1)!}\left(\psi^{(p)}\left(\frac nm\right)
-2\gamma\psi^{(p-1)}\left(\frac nm\right)-\sum_
{k=0}^{p-1} \binom{p-1}{k}\psi^{(k)}\left(\frac nm\right)\psi^{(p-1-k)}\left(\frac nm\right)
\right).
\end{equation}
\end{corollary}
\begin{proof}
This follows from \cref{thm:thm2} after setting $t = n/m$.
\end{proof}
Numerous examples of Theorem \ref{thm:thm2} and Corollary \ref{cor:cor1} are presented in Appendix \ref{app:catalog}.

\subsection{Mixed denominators}
Combining the residues at each pole finally gives the following general result.
\begin{theorem} \label{thm:thm3}
Let $C_{i,j}$ be the coefficients of the partial fraction decomposition
\begin{equation}
R(k) = \sum_{j=1}^r\sum_{i=1}^{p_j}\frac{C_{i,j}}{(k+t_j)^i},
\end{equation}
with distinct~$t_j$, $C_{p_j,j}\neq0$ and integers $p_j \geq
1$ whose sum is at least~2. Then
\begin{align}
\sum_{k=1}^\infty R(k){H_k}=\sum_{j=1}^r\sum_{i=1}^{p_j}C_{i,j}T(t_j,i),
\label{form:theorem4a}
\end{align}
with $T$ the function from \cref{eq:Ttp}.
\end{theorem}

\noindent
As a typical illustration, here is an extension of 
\cref{thm:thm3}.
\begin{corollary} \label{cor:cor2}
For integers $m, n \geq 1, \gcd (m, n) = 1, q \geq 1$ and $p \geq q+2$, and for the $T$ notation as defined in \cref{eq:Ttp},
\[\sum_{k=1}^\infty \frac{k^q H_k}{(mk+n)^p} =\frac1{m^p} \sum_
{j=0}^q 
(-1)^j
\binom{q}{j} \left(\frac nm\right)^j T \left(\frac nm,p-q+j\right).\]
\end{corollary}
\begin{proof} By \cref{thm:thm3}, it is sufficient to compute the
coefficients of the partial fraction decomposition of $k^q/(mk+n)^p$
with respect to~$k$. The result then follows from
\[\frac{k^q}{(mk+n)^p}=\frac{m^{-p}k^q}{(k+n/m)^p}=\frac{m^{-p}}{(k+n/m)^
{p-q}}\left (1- \frac{n/m}
{k+n/m}\right)^q=m^{-p}\sum_{j=0}^q (-1)^j \binom{q}{j} \frac{(n/m)^j}
{
(k+n/m)^
{p-q+j}}.\qedhere\]
\end{proof}
\begin{example}
Example \ref{form:illcor3} can be computed by hand from
the partial fraction decomposition 
\[\frac{1}{(2k+1)^3(3k+1)^2}=
\frac4{(2k+1)^3}
+\frac{24}{(2k+1)^2}
+\frac{108}{2k+1}
+\frac{27}{(3k+1)^2}
-\frac{162}{3k+1},
\]
which is easy to compute and even easier to check. 
\end{example}
Numerous other examples of Theorem \ref{thm:thm3} are presented in Appendix \ref{app:catalog}.

\section{Numerical computation of mixed-denominator Euler sums} 
\label{sec:direct}

The research leading to the results in the previous section relied on direct high-precision numerical evaluations of specific Euler sums, combined with an integer relation computation to produce the coefficients of the right-hand side terms. Also, even with the above results in hand, we have found that it is often easier to use numerical methods to obtain the formulas. In fact, the catalog of formulas in Appendix \ref{app:catalog} was obtained by this process, after checking in each case that the numerical value matches the result given by the above theorems to high precision.

The numerical scheme we employed here is a variation of a scheme described in \cite{McPBai2025}. In very brief summary, the overall strategy is to compute a large number ($10^8$) terms of the Euler sum explicitly to high precision, then employ the Euler-Maclaurin summation formula \cite[pg.~285]{Atkinson1989} twice to sum the tail of the series. The specific algorithm we employed resulted in an approximation correct to within roughly $10^{-290}$.  In our implementation, computed results were typically accurate to roughly 280 digits. 

Once a high-precision numerical value was obtained, we employed the multipair PSLQ integer relation algorithm \cite{Bailey2000}, applied to the numerical values of the sum and the hypothesized right-hand side constants, to obtain the rational coefficients in the catalogued formulas below.

\section{Computing digamma and polygamma} \label{sec:polygamma}

The formulas given above involve Euler's constant, the digamma function and the polygamma function. \emph{Maple} and \emph{Mathematica} can evaluate these to arbitrary precision, but not all researchers have access to this commercial software, and others may wish to explore these relations with custom code.

In our computations, we employed formulas and algorithms presented in \cite{mpfun2020}, to which the reader is referred for full details. In very brief summary, we employed this scheme for digamma: For a modest value of $x$, first apply the recursion $\psi(x+1) = \psi(x) + 1/x$, repeating $0.45B$ times, where $B$ is the working precision in bits, to shift the argument to a larger value. Then employ this formula \cite[5.11.2]{dlmf}:
\begin{align}
\psi(x) &\approx \log(x) - \frac{1}{2x} - \sum_{k=1}^n \frac{B_{2k}}{2k x^{2k}}, \label{form:digamma}
\end{align}
where $n = 2 D$, $D$ is the precision in digits, and $\{B_{2k}\}$ is a set of precomputed even Bernoulli numbers.

A similar scheme for polygamma is the following: For $\psi^{(q)}(z)$ with integer $q \geq 1$, first reduce $z$ to the range $(0,1]$ by applying the recurrence $\psi^{(q)}(z) = \psi^{(q)}(z+1) - (-1)^q q! z^{-q-1}$. Then polygamma can be evaluated from the Hurwitz zeta function via the formula \cite{wiki-polygamma} $\psi^{(q)}(z) = (-1)^{q+1} q! \zeta(q+1, z)$. To compute the Hurwitz zeta function, select an integer $q > 0.6D$, where $D$ is the precision in digits. Let $n = 2 D$ as above. Then
\begin{align}
\zeta(s,a) &\approx \sum_{k=0}^{q-1} \frac{1}{(a+k)^s} + \frac{1}{(s-1)(a+q)^{s-1}} + \frac{1}{2(a+q)^s} + \sum_{k=1}^n \frac{B_{2k} \, s (s+1) \cdots (s+2k-2)}{(2k)! (a+q)^{s+2k-1}}.
\end{align}
The above algorithms require a precomputed set of even Bernoulli numbers $B_{2k}$ of size $n$. These can be computed efficiently by a scheme presented in \cite{mpfun2020}.

\section{Conclusion}

In this paper, we have presented explicit formulas and techniques for analytically evaluating Euler sums of a general class, namely sums of the form $\sum_{k=1}^\infty R(k) H_k$, where $R(k)$ is any rational function whose denominator degree is at least two greater than the numerator degree. We have also briefly presented computational techniques to compute these sums numerically to very high precision, which enabled us to confirm the analytic formulas in hundreds of specific cases.

One fair question here is whether the analytic formulas given by the above theorems truly add value: Don't they just replace one summation (the left-hand side) with several more (one for each digamma and polygamma function evaluation)?  The key fact here is that digamma and polygamma function evaluations, if performed using the fast algorithms mentioned in Section \ref{sec:polygamma}, are typically many times faster than evaluating Euler sums directly using the method given in Section \ref{sec:direct}. For example, we found that numerically evaluating the Euler sum $\sum_{k=1}^\infty H_k/(3k+2)^4 = 0.002220075526\ldots$ to 280-digit precision using the fast digamma and polygamma formulas was approximately 30,000 times faster than using the direct scheme.

So there is value in these formulas beyond mere mathematical elegance. Also, the results presented here may be of use in accelerating sums with a leading term of logarithmic form in the summand numerator.

In any event, it is clear there is still much that can be done. For example, using other residue representations, one can obtain an algorithm along the lines of Algorithm~\ref{algo:algo1} for many sums of the types
\begin{align}
\sum_{k=1}^\infty \frac{H_k^2}{(m k + n)^p}, \quad \sum_{k=1}^\infty \frac{(-1)^k H_k}{(m k + n)^p}, \quad \sum_{k=1}^\infty \frac{H_{2k}}{(m k + n)^p},
\end{align}
for positive integers $m \geq 3, \gcd(m,n) = 1$ and $p \geq 2$.

\appendix

\section{Maple code for Algorithm 
\ref{algo:algo1}}
\label{app:maple_code}

The Maple code below has been successfully tested on the examples of this article. We have also translated  this code into Mathematica, using one of the currently available large language models (LLMs), and found that the resulting code passed a similar set of tests. But in general one should be cautious with such translations, since the handling of algebraic numbers can be very different in different computer algebra systems, and, based on our experience, not all LLMs do this translation correctly.

\newpage

\begin{verbatim}
# Input: 
#   . R a rational function
#   . k its variable
# Output: 
#   $\sum_{k=1}^\infty{R(k)H_k}$,
# where $H_k$ is the $k$th harnomic number $1+1/2+\dots+1/k$.
harmonicsum:=proc(R,k)
local r,numr,denr,ker,facts,factmult,fact,pole,rest,restfact,S,res,val,mult,expansionr,alpha;
    r:=normal(R);   # reduce to the same denominator
    numr:=numer(r); denr:=denom(r);
    if degree(denr,k)-degree(numr,k)<-2 then error "sum not convergent" fi;
    ker:=Psi(-k)+gamma;
    facts:=factors(denr)[2];
    for factmult in facts do # deal with factors of the denominator one by one
        fact,mult:=op(factmult); # factor and multiplicity
        pole:=RootOf(fact,k);
        if type(pole,posint) then error "infinite summand" fi;
        rest:=normal(r*fact^mult);
        # expansion of $r$ at the pole
        restfact:=subs(alpha=pole,series((fact-eval(fact,k=alpha))/(k-alpha),k=alpha,mult+1));
        expansionr:=series(rest/restfact^mult/(k-pole)^mult,k=pole,mult+1);
        # simplify coefficients for nonrational poles
        if has(pole,RootOf) then expansionr:=map(evala,expansionr) fi;
        # expansion of $r'(k)(\psi(-k)+\gamma)-r(k)(\psi(-k)+\gamma)^2$ at this pole
        S:=series(diff(expansionr,k)*ker-expansionr*ker^2,k=pole,mult+1);
        # extract residue
        res:=coeff(S,k-pole,-1)/2;
        # sum over all conjugate roots of this factor
        if has(res,RootOf) then val[factmult]:=convert([allvalues(res)],`+`)
        else val[factmult]:=res
        fi
    od;
    collect(add(val[factmult],factmult=facts),Psi,'distributed')
end:
\end{verbatim}

\section{Catalog of formulas} \label{app:catalog}

We present here a selection of formulas that we have produced according to the above techniques. In each case, we computed the Euler sum numerically in two ways: (a) by applying the theorems of Section \ref{sec:theorems}, together with the numerical algorithms given in Section \ref{sec:polygamma}, and (b) by the direct scheme described in Section \ref{sec:direct}. After verifying that the two numerical values were in agreement, the multipair PSLQ algorithm, applied to the computed value and the hypothesized right-hand side constants, then found the rational coefficients and generated the formulas below. 

\emph{To guard against the possibility of transcription errors, the LaTeX code for these formulas was produced by a program directly from the computer output, and is included here without any alteration.}

\subsection{Theorem \ref{thm:thm2} (Corollary \ref{cor:cor1}) formulas} \label{sec:catcor1}

\begin{align}
\sum_{k=1}^\infty \frac{H_k}{(2k+1)^2} &= \frac{1}{8}\left( -1\psi(2,1/2) +2\gamma\psi(1,1/2) +2\psi(0,1/2)\psi(1,1/2)\right) \\
\sum_{k=1}^\infty \frac{H_k}{(2k+1)^3} &= \frac{ 1}{16}\left( -1\gamma\psi(2,1/2) -1\psi(0,1/2)\psi(2,1/2) +1\psi(1,1/2)\psi(1,1/2)\right) \\
\sum_{k=1}^\infty \frac{H_k}{(2k+1)^4} &= \frac{1}{192}\left( -1\psi(4,1/2) +2\gamma\psi(3,1/2) +2\psi(0,1/2)\psi(3,1/2) +6\psi(1,1/2)\psi(2,1/2)\right) \\
\sum_{k=1}^\infty \frac{H_k}{(2k+1)^6} &= \frac{ 1}{15360}\left( -1\psi(6,1/2) +2\gamma\psi(5,1/2) +2\psi(0,1/2)\psi(5,1/2) +10\psi(1,1/2)\psi(4,1/2)  \right. \nonumber \\ &\left. \hspace{1em}
+20\psi(2,1/2)\psi(3,1/2)\right) \\
\sum_{k=1}^\infty \frac{H_k}{(2k+1)^7} &= \frac{1}{368640}\left(  1\psi(7,1/2) -4\gamma\psi(6,1/2) -4\psi(0,1/2)\psi(6,1/2) -60\psi(2,1/2)\psi(4,1/2)\right) \\
\sum_{k=1}^\infty \frac{H_k}{(2k+1)^8} &= \frac{1}{2580480}\left( -1\psi(8,1/2) +2\gamma\psi(7,1/2) +2\psi(0,1/2)\psi(7,1/2) +14\psi(1,1/2)\psi(6,1/2)  \right. \nonumber \\ &\left. \hspace{1em}
+42\psi(2,1/2)\psi(5,1/2) +70\psi(3,1/2)\psi(4,1/2)\right) \\
\sum_{k=1}^\infty \frac{H_k}{(3k+1)^2} &= \frac{1}{18}\left( -1\psi(2,1/3) +2\gamma\psi(1,1/3) +2\psi(0,1/3)\psi(1,1/3)\right) \\
\sum_{k=1}^\infty \frac{H_k}{(3k+1)^3} &= \frac{ 1}{108}\left(  1\psi(3,1/3) -2\gamma\psi(2,1/3) -2\psi(0,1/3)\psi(2,1/3) -2\psi(1,1/3)\psi(1,1/3)\right) \\
\sum_{k=1}^\infty \frac{H_k}{(3k+1)^4} &= \frac{ 1}{972}\left( -1\psi(4,1/3) +2\gamma\psi(3,1/3) +2\psi(0,1/3)\psi(3,1/3) +6\psi(1,1/3)\psi(2,1/3)\right) \\
\sum_{k=1}^\infty \frac{H_k}{(3k+1)^5} &= \frac{1}{11664}\left(  1\psi(5,1/3) -2\gamma\psi(4,1/3) -2\psi(0,1/3)\psi(4,1/3) -8\psi(1,1/3)\psi(3,1/3)  \right. \nonumber \\ &\left. \hspace{1em}
-6\psi(2,1/3)\psi(2,1/3)\right) \\
\sum_{k=1}^\infty \frac{H_k}{(3k+1)^6} &= \frac{ 1}{174960}\left( -1\psi(6,1/3) +2\gamma\psi(5,1/3) +2\psi(0,1/3)\psi(5,1/3) +10\psi(1,1/3)\psi(4,1/3)  \right. \nonumber \\ &\left. \hspace{1em}
+20\psi(2,1/3)\psi(3,1/3)\right) \\
\sum_{k=1}^\infty \frac{H_k}{(3k+1)^7} &= \frac{1}{3149280}\left(  1\psi(7,1/3) -2\gamma\psi(6,1/3) -2\psi(0,1/3)\psi(6,1/3) -12\psi(1,1/3)\psi(5,1/3)  \right. \nonumber \\ &\left. \hspace{1em}
-30\psi(2,1/3)\psi(4,1/3) -20\psi(3,1/3)\psi(3,1/3)\right) \\
\sum_{k=1}^\infty \frac{H_k}{(3k+1)^8} &= \frac{ 1}{66134880}\left( -1\psi(8,1/3) +2\gamma\psi(7,1/3) +2\psi(0,1/3)\psi(7,1/3) +14\psi(1,1/3)\psi(6,1/3)  \right. \nonumber \\ &\left. \hspace{1em}
+42\psi(2,1/3)\psi(5,1/3) +70\psi(3,1/3)\psi(4,1/3)\right)
\end{align}
 
\begin{align}
\sum_{k=1}^\infty \frac{H_k}{(3k+2)^2} &= \frac{ 1}{18}\left( -1\psi(2,2/3) +2\gamma\psi(1,2/3) +2\psi(0,2/3)\psi(1,2/3)\right) \\
\sum_{k=1}^\infty \frac{H_k}{(3k+2)^3} &= \frac{ 1}{108}\left(  1\psi(3,2/3) -2\gamma\psi(2,2/3) -2\psi(0,2/3)\psi(2,2/3) -2\psi(1,2/3)\psi(1,2/3)\right) \\
\sum_{k=1}^\infty \frac{H_k}{(3k+2)^4} &= \frac{ 1}{972}\left( -1\psi(4,2/3) +2\gamma\psi(3,2/3) +2\psi(0,2/3)\psi(3,2/3) +6\psi(1,2/3)\psi(2,2/3)\right) \\
\sum_{k=1}^\infty \frac{H_k}{(3k+2)^5} &= \frac{1}{11664}\left(  1\psi(5,2/3) -2\gamma\psi(4,2/3) -2\psi(0,2/3)\psi(4,2/3) -8\psi(1,2/3)\psi(3,2/3)  \right. \nonumber \\ &\left. \hspace{1em}
-6\psi(2,2/3)\psi(2,2/3)\right) \\
\sum_{k=1}^\infty \frac{H_k}{(3k+2)^6} &= \frac{ 1}{174960}\left( -1\psi(6,2/3) +2\gamma\psi(5,2/3) +2\psi(0,2/3)\psi(5,2/3) +10\psi(1,2/3)\psi(4,2/3)  \right. \nonumber \\ &\left. \hspace{1em}
+20\psi(2,2/3)\psi(3,2/3)\right) \\
\sum_{k=1}^\infty \frac{H_k}{(3k+2)^7} &= \frac{ 1}{3149280}\left(  1\psi(7,2/3) -2\gamma\psi(6,2/3) -2\psi(0,2/3)\psi(6,2/3) -12\psi(1,2/3)\psi(5,2/3)  \right. \nonumber \\ &\left. \hspace{1em}
-30\psi(2,2/3)\psi(4,2/3) -20\psi(3,2/3)\psi(3,2/3)\right) \\
\sum_{k=1}^\infty \frac{H_k}{(3k+2)^8} &= \frac{1}{66134880}\left( -1\psi(8,2/3) +2\gamma\psi(7,2/3) +2\psi(0,2/3)\psi(7,2/3) +14\psi(1,2/3)\psi(6,2/3)  \right. \nonumber \\ &\left. \hspace{1em}
+42\psi(2,2/3)\psi(5,2/3) +70\psi(3,2/3)\psi(4,2/3)\right) \\
\sum_{k=1}^\infty \frac{H_k}{(4k+1)^2} &= \frac{ 1}{32}\left( -1\psi(2,1/4) +2\gamma\psi(1,1/4) +2\psi(0,1/4)\psi(1,1/4)\right) \\
\sum_{k=1}^\infty \frac{H_k}{(4k+1)^3} &= \frac{ 1}{256}\left(  1\psi(3,1/4) -2\gamma\psi(2,1/4) -2\psi(0,1/4)\psi(2,1/4) -2\psi(1,1/4)\psi(1,1/4)\right) \\
\sum_{k=1}^\infty \frac{H_k}{(4k+1)^4} &= \frac{ 1}{3072}\left( -1\psi(4,1/4) +2\gamma\psi(3,1/4) +2\psi(0,1/4)\psi(3,1/4) +6\psi(1,1/4)\psi(2,1/4)\right) \\
\sum_{k=1}^\infty \frac{H_k}{(4k+1)^5} &= \frac{ 1}{49152}\left(  1\psi(5,1/4) -2\gamma\psi(4,1/4) -2\psi(0,1/4)\psi(4,1/4) -8\psi(1,1/4)\psi(3,1/4)  \right. \nonumber \\ &\left. \hspace{1em}
-6\psi(2,1/4)\psi(2,1/4)\right) \\
\sum_{k=1}^\infty \frac{H_k}{(4k+1)^6} &= \frac{1}{983040}\left( -1\psi(6,1/4) +2\gamma\psi(5,1/4) +2\psi(0,1/4)\psi(5,1/4) +10\psi(1,1/4)\psi(4,1/4)  \right. \nonumber \\ &\left. \hspace{1em}
+20\psi(2,1/4)\psi(3,1/4)\right) \\
\sum_{k=1}^\infty \frac{H_k}{(4k+1)^7} &= \frac{ 1}{23592960}\left(  1\psi(7,1/4) -2\gamma\psi(6,1/4) -2\psi(0,1/4)\psi(6,1/4) -12\psi(1,1/4)\psi(5,1/4)  \right. \nonumber \\ &\left. \hspace{1em}
-30\psi(2,1/4)\psi(4,1/4) -20\psi(3,1/4)\psi(3,1/4)\right)
\end{align}
 
\begin{align}
\sum_{k=1}^\infty \frac{H_k}{(4k+1)^8} &= \frac{ 1}{660602880}\left( -1\psi(8,1/4) +2\gamma\psi(7,1/4) +2\psi(0,1/4)\psi(7,1/4) +14\psi(1,1/4)\psi(6,1/4)  \right. \nonumber \\ &\left. \hspace{1em}
+42\psi(2,1/4)\psi(5,1/4) +70\psi(3,1/4)\psi(4,1/4)\right) \\
\sum_{k=1}^\infty \frac{H_k}{(4k+3)^2} &= \frac{1}{32}\left( -1\psi(2,3/4) +2\gamma\psi(1,3/4) +2\psi(0,3/4)\psi(1,3/4)\right) \\
\sum_{k=1}^\infty \frac{H_k}{(4k+3)^3} &= \frac{1}{256}\left(  1\psi(3,3/4) -2\gamma\psi(2,3/4) -2\psi(0,3/4)\psi(2,3/4) -2\psi(1,3/4)\psi(1,3/4)\right) \\
\sum_{k=1}^\infty \frac{H_k}{(4k+3)^4} &= \frac{ 1}{3072}\left( -1\psi(4,3/4) +2\gamma\psi(3,3/4) +2\psi(0,3/4)\psi(3,3/4) +6\psi(1,3/4)\psi(2,3/4)\right) \\
\sum_{k=1}^\infty \frac{H_k}{(4k+3)^5} &= \frac{1}{49152}\left(  1\psi(5,3/4) -2\gamma\psi(4,3/4) -2\psi(0,3/4)\psi(4,3/4) -8\psi(1,3/4)\psi(3,3/4)  \right. \nonumber \\ &\left. \hspace{1em}
-6\psi(2,3/4)\psi(2,3/4)\right) \\
\sum_{k=1}^\infty \frac{H_k}{(4k+3)^6} &= \frac{ 1}{983040}\left( -1\psi(6,3/4) +2\gamma\psi(5,3/4) +2\psi(0,3/4)\psi(5,3/4) +10\psi(1,3/4)\psi(4,3/4)  \right. \nonumber \\ &\left. \hspace{1em}
+20\psi(2,3/4)\psi(3,3/4)\right) \\
\sum_{k=1}^\infty \frac{H_k}{(4k+3)^7} &= \frac{1}{23592960}\left(  1\psi(7,3/4) -2\gamma\psi(6,3/4) -2\psi(0,3/4)\psi(6,3/4) -12\psi(1,3/4)\psi(5,3/4)  \right. \nonumber \\ &\left. \hspace{1em}
-30\psi(2,3/4)\psi(4,3/4) -20\psi(3,3/4)\psi(3,3/4)\right) \\
\sum_{k=1}^\infty \frac{H_k}{(4k+3)^8} &= \frac{ 1}{660602880}\left( -1\psi(8,3/4) +2\gamma\psi(7,3/4) +2\psi(0,3/4)\psi(7,3/4) +14\psi(1,3/4)\psi(6,3/4)  \right. \nonumber \\ &\left. \hspace{1em}
+42\psi(2,3/4)\psi(5,3/4) +70\psi(3,3/4)\psi(4,3/4)\right)
\end{align}

\subsection{Theorem \ref{thm:thm3} formulas} \label{sec:catcor2}

\begin{align}
\sum_{k=1}^\infty \frac{H_k}{(2k+1)(3k+1)} &= \frac{1}{2}\left( -1\psi(1,1/2) +2\gamma\psi(0,1/2) +1\psi(0,1/2)\psi(0,1/2) +1\psi(1,1/3) -2\gamma\psi(0,1/3)  \right. \nonumber \\ &\left. \hspace{1em}
-1\psi(0,1/3)\psi(0,1/3)\right) \\
\sum_{k=1}^\infty \frac{H_k}{(2k+1)(3k+1)^2} &= \frac{-1}{6}\left( -6\psi(1,1/2) +12\gamma\psi(0,1/2) +6\psi(0,1/2)\psi(0,1/2) +6\psi(1,1/3) -12\gamma\psi(0,1/3)  \right. \nonumber \\ &\left. \hspace{1em}
-6\psi(0,1/3)\psi(0,1/3) +1\psi(2,1/3) -2\gamma\psi(1,1/3) -2\psi(0,1/3)\psi(1,1/3)\right) \\
\sum_{k=1}^\infty \frac{H_k}{(2k+1)(3k+1)^3} &= \frac{-1}{36}\left(  72\psi(1,1/2) -144\gamma\psi(0,1/2) -72\psi(0,1/2)\psi(0,1/2) -72\psi(1,1/3) +144\gamma\psi(0,1/3)  \right. \nonumber \\ &\left. \hspace{1em}
+72\psi(0,1/3)\psi(0,1/3) -12\psi(2,1/3) +24\gamma\psi(1,1/3) +24\psi(0,1/3)\psi(1,1/3)  \right. \nonumber \\ &\left. \hspace{1em}
-1\psi(3,1/3) +2\gamma\psi(2,1/3) +2\psi(0,1/3)\psi(2,1/3) +2\psi(1,1/3)\psi(1,1/3)\right) \\
\sum_{k=1}^\infty \frac{H_k}{(2k+1)(3k+2)} &= \frac{1}{2}\left(  1\psi(1,1/2) -2\gamma\psi(0,1/2) -1\psi(0,1/2)\psi(0,1/2) -1\psi(1,2/3) +2\gamma\psi(0,2/3)  \right. \nonumber \\ &\left. \hspace{1em}
+1\psi(0,2/3)\psi(0,2/3)\right) \\
\sum_{k=1}^\infty \frac{H_k}{(2k+1)(3k+2)^2} &= \frac{-1}{6}\left( -6\psi(1,1/2) +12\gamma\psi(0,1/2) +6\psi(0,1/2)\psi(0,1/2) +6\psi(1,2/3) -12\gamma\psi(0,2/3)  \right. \nonumber \\ &\left. \hspace{1em}
-6\psi(0,2/3)\psi(0,2/3) -1\psi(2,2/3) +2\gamma\psi(1,2/3) +2\psi(0,2/3)\psi(1,2/3)\right) \\
\sum_{k=1}^\infty \frac{H_k}{(2k+1)(3k+2)^3} &= \frac{-1}{36}\left( -72\psi(1,1/2) +144\gamma\psi(0,1/2) +72\psi(0,1/2)\psi(0,1/2) +72\psi(1,2/3) -144\gamma\psi(0,2/3)  \right. \nonumber \\ &\left. \hspace{1em}
-72\psi(0,2/3)\psi(0,2/3) -12\psi(2,2/3) +24\gamma\psi(1,2/3) +24\psi(0,2/3)\psi(1,2/3)  \right. \nonumber \\ &\left. \hspace{1em}
+1\psi(3,2/3) -2\gamma\psi(2,2/3) -2\psi(0,2/3)\psi(2,2/3) -2\psi(1,2/3)\psi(1,2/3)\right) \\
\sum_{k=1}^\infty \frac{H_k}{(2k+1)(4k+1)} &= \frac{-1}{4}\left(  1\psi(1,1/2) -2\gamma\psi(0,1/2) -1\psi(0,1/2)\psi(0,1/2) -1\psi(1,1/4) +2\gamma\psi(0,1/4)  \right. \nonumber \\ &\left. \hspace{1em}
+1\psi(0,1/4)\psi(0,1/4)\right) \\
\sum_{k=1}^\infty \frac{H_k}{(2k+1)(4k+1)^2} &= \frac{-1}{16}\left( -4\psi(1,1/2) +8\gamma\psi(0,1/2) +4\psi(0,1/2)\psi(0,1/2) +4\psi(1,1/4) -8\gamma\psi(0,1/4)  \right. \nonumber \\ &\left. \hspace{1em}
-4\psi(0,1/4)\psi(0,1/4) +1\psi(2,1/4) -2\gamma\psi(1,1/4) -2\psi(0,1/4)\psi(1,1/4)\right) \\
\sum_{k=1}^\infty \frac{H_k}{(2k+1)(4k+1)^3} &= \frac{-1}{128}\left(  32\psi(1,1/2) -64\gamma\psi(0,1/2) -32\psi(0,1/2)\psi(0,1/2) -32\psi(1,1/4) +64\gamma\psi(0,1/4)  \right. \nonumber \\ &\left. \hspace{1em}
+32\psi(0,1/4)\psi(0,1/4) -8\psi(2,1/4) +16\gamma\psi(1,1/4) +16\psi(0,1/4)\psi(1,1/4)  \right. \nonumber \\ &\left. \hspace{1em}
-1\psi(3,1/4) +2\gamma\psi(2,1/4) +2\psi(0,1/4)\psi(2,1/4) +2\psi(1,1/4)\psi(1,1/4)\right) \\
\sum_{k=1}^\infty \frac{H_k}{(2k+1)(4k+3)} &= \frac{-1}{4}\left( -1\psi(1,1/2) +2\gamma\psi(0,1/2) +1\psi(0,1/2)\psi(0,1/2) +1\psi(1,3/4) -2\gamma\psi(0,3/4)  \right. \nonumber \\ &\left. \hspace{1em}
-1\psi(0,3/4)\psi(0,3/4)\right)
\end{align}
 
\begin{align}
\sum_{k=1}^\infty \frac{H_k}{(2k+1)(4k+3)^2} &= \frac{1}{16}\left(  4\psi(1,1/2) -8\gamma\psi(0,1/2) -4\psi(0,1/2)\psi(0,1/2) -4\psi(1,3/4) +8\gamma\psi(0,3/4)  \right. \nonumber \\ &\left. \hspace{1em}
+4\psi(0,3/4)\psi(0,3/4) +1\psi(2,3/4) -2\gamma\psi(1,3/4) -2\psi(0,3/4)\psi(1,3/4)\right) \\
\sum_{k=1}^\infty \frac{H_k}{(2k+1)(4k+3)^3} &= \frac{1}{128}\left(  32\psi(1,1/2) -64\gamma\psi(0,1/2) -32\psi(0,1/2)\psi(0,1/2) -32\psi(1,3/4) +64\gamma\psi(0,3/4)  \right. \nonumber \\ &\left. \hspace{1em}
+32\psi(0,3/4)\psi(0,3/4) +8\psi(2,3/4) -16\gamma\psi(1,3/4) -16\psi(0,3/4)\psi(1,3/4)  \right. \nonumber \\ &\left. \hspace{1em}
-1\psi(3,3/4) +2\gamma\psi(2,3/4) +2\psi(0,3/4)\psi(2,3/4) +2\psi(1,3/4)\psi(1,3/4)\right) \\
\sum_{k=1}^\infty \frac{H_k}{(2k+1)^2(3k+1)} &= \frac{-1}{4}\left(  6\psi(1,1/2) -12\gamma\psi(0,1/2) -6\psi(0,1/2)\psi(0,1/2) -1\psi(2,1/2) +2\gamma\psi(1,1/2)  \right. \nonumber \\ &\left. \hspace{1em}
+2\psi(0,1/2)\psi(1,1/2) -6\psi(1,1/3) +12\gamma\psi(0,1/3) +6\psi(0,1/3)\psi(0,1/3)\right) \\
\sum_{k=1}^\infty \frac{H_k}{(2k+1)^2(3k+1)^2} &= \frac{-1}{2}\left( -12\psi(1,1/2) +24\gamma\psi(0,1/2) +12\psi(0,1/2)\psi(0,1/2) +1\psi(2,1/2) -2\gamma\psi(1,1/2)  \right. \nonumber \\ &\left. \hspace{1em}
-2\psi(0,1/2)\psi(1,1/2) +12\psi(1,1/3) -24\gamma\psi(0,1/3) -12\psi(0,1/3)\psi(0,1/3)  \right. \nonumber \\ &\left. \hspace{1em}
+1\psi(2,1/3) -2\gamma\psi(1,1/3) -2\psi(0,1/3)\psi(1,1/3)\right) \\
\sum_{k=1}^\infty \frac{H_k}{(2k+1)^2(3k+1)^3} &= \frac{1}{12}\left( -216\psi(1,1/2) +432\gamma\psi(0,1/2) +216\psi(0,1/2)\psi(0,1/2) +12\psi(2,1/2)  \right. \nonumber \\ &\left. \hspace{1em}
-24\gamma\psi(1,1/2) -24\psi(0,1/2)\psi(1,1/2) +216\psi(1,1/3) -432\gamma\psi(0,1/3)  \right. \nonumber \\ &\left. \hspace{1em}
-216\psi(0,1/3)\psi(0,1/3) +24\psi(2,1/3) -48\gamma\psi(1,1/3) -48\psi(0,1/3)\psi(1,1/3)  \right. \nonumber \\ &\left. \hspace{1em}
+1\psi(3,1/3) -2\gamma\psi(2,1/3) -2\psi(0,1/3)\psi(2,1/3) -2\psi(1,1/3)\psi(1,1/3)\right) \\
\sum_{k=1}^\infty \frac{H_k}{(2k+1)^2(3k+2)} &= \frac{-1}{4}\left(  6\psi(1,1/2) -12\gamma\psi(0,1/2) -6\psi(0,1/2)\psi(0,1/2) +1\psi(2,1/2) -2\gamma\psi(1,1/2)  \right. \nonumber \\ &\left. \hspace{1em}
-2\psi(0,1/2)\psi(1,1/2) -6\psi(1,2/3) +12\gamma\psi(0,2/3) +6\psi(0,2/3)\psi(0,2/3)\right) \\
\sum_{k=1}^\infty \frac{H_k}{(2k+1)^2(3k+2)^2} &= \frac{1}{2}\left( -12\psi(1,1/2) +24\gamma\psi(0,1/2) +12\psi(0,1/2)\psi(0,1/2) -1\psi(2,1/2) +2\gamma\psi(1,1/2)  \right. \nonumber \\ &\left. \hspace{1em}
+2\psi(0,1/2)\psi(1,1/2) +12\psi(1,2/3) -24\gamma\psi(0,2/3) -12\psi(0,2/3)\psi(0,2/3)  \right. \nonumber \\ &\left. \hspace{1em}
-1\psi(2,2/3) +2\gamma\psi(1,2/3) +2\psi(0,2/3)\psi(1,2/3)\right) \\
\sum_{k=1}^\infty \frac{H_k}{(2k+1)^2(3k+2)^3} &= \frac{1}{12}\left( -216\psi(1,1/2) +432\gamma\psi(0,1/2) +216\psi(0,1/2)\psi(0,1/2) -12\psi(2,1/2)  \right. \nonumber \\ &\left. \hspace{1em}
+24\gamma\psi(1,1/2) +24\psi(0,1/2)\psi(1,1/2) +216\psi(1,2/3) -432\gamma\psi(0,2/3)  \right. \nonumber \\ &\left. \hspace{1em}
-216\psi(0,2/3)\psi(0,2/3) -24\psi(2,2/3) +48\gamma\psi(1,2/3) +48\psi(0,2/3)\psi(1,2/3)  \right. \nonumber \\ &\left. \hspace{1em}
+1\psi(3,2/3) -2\gamma\psi(2,2/3) -2\psi(0,2/3)\psi(2,2/3) -2\psi(1,2/3)\psi(1,2/3)\right) \\
\sum_{k=1}^\infty \frac{H_k}{(2k+1)^2(4k+1)} &= \frac{1}{8}\left( -4\psi(1,1/2) +8\gamma\psi(0,1/2) +4\psi(0,1/2)\psi(0,1/2) +1\psi(2,1/2) -2\gamma\psi(1,1/2)  \right. \nonumber \\ &\left. \hspace{1em}
-2\psi(0,1/2)\psi(1,1/2) +4\psi(1,1/4) -8\gamma\psi(0,1/4) -4\psi(0,1/4)\psi(0,1/4)\right)
\end{align}
 
\begin{align}
\sum_{k=1}^\infty \frac{H_k}{(2k+1)^2(4k+1)^2} &= \frac{-1}{8}\left( -8\psi(1,1/2) +16\gamma\psi(0,1/2) +8\psi(0,1/2)\psi(0,1/2) +1\psi(2,1/2) -2\gamma\psi(1,1/2)  \right. \nonumber \\ &\left. \hspace{1em}
-2\psi(0,1/2)\psi(1,1/2) +8\psi(1,1/4) -16\gamma\psi(0,1/4) -8\psi(0,1/4)\psi(0,1/4) +1\psi(2,1/4)  \right. \nonumber \\ &\left. \hspace{1em}
-2\gamma\psi(1,1/4) -2\psi(0,1/4)\psi(1,1/4)\right) \\
\sum_{k=1}^\infty \frac{H_k}{(2k+1)^2(4k+1)^3} &= \frac{1}{64}\left( -96\psi(1,1/2) +192\gamma\psi(0,1/2) +96\psi(0,1/2)\psi(0,1/2) +8\psi(2,1/2) -16\gamma\psi(1,1/2)  \right. \nonumber \\ &\left. \hspace{1em}
-16\psi(0,1/2)\psi(1,1/2) +96\psi(1,1/4) -192\gamma\psi(0,1/4) -96\psi(0,1/4)\psi(0,1/4)  \right. \nonumber \\ &\left. \hspace{1em}
+16\psi(2,1/4) -32\gamma\psi(1,1/4) -32\psi(0,1/4)\psi(1,1/4) +1\psi(3,1/4) -2\gamma\psi(2,1/4)  \right. \nonumber \\ &\left. \hspace{1em}
-2\psi(0,1/4)\psi(2,1/4) -2\psi(1,1/4)\psi(1,1/4)\right) \\
\sum_{k=1}^\infty \frac{H_k}{(2k+1)^2(4k+3)} &= \frac{1}{8}\left( -4\psi(1,1/2) +8\gamma\psi(0,1/2) +4\psi(0,1/2)\psi(0,1/2) -1\psi(2,1/2) +2\gamma\psi(1,1/2)  \right. \nonumber \\ &\left. \hspace{1em}
+2\psi(0,1/2)\psi(1,1/2) +4\psi(1,3/4) -8\gamma\psi(0,3/4) -4\psi(0,3/4)\psi(0,3/4)\right) \\
\sum_{k=1}^\infty \frac{H_k}{(2k+1)^2(4k+3)^2} &= \frac{-1}{8}\left(  8\psi(1,1/2) -16\gamma\psi(0,1/2) -8\psi(0,1/2)\psi(0,1/2) +1\psi(2,1/2) -2\gamma\psi(1,1/2)  \right. \nonumber \\ &\left. \hspace{1em}
-2\psi(0,1/2)\psi(1,1/2) -8\psi(1,3/4) +16\gamma\psi(0,3/4) +8\psi(0,3/4)\psi(0,3/4) +1\psi(2,3/4)  \right. \nonumber \\ &\left. \hspace{1em}
-2\gamma\psi(1,3/4) -2\psi(0,3/4)\psi(1,3/4)\right) \\
\sum_{k=1}^\infty \frac{H_k}{(2k+1)^2(4k+3)^3} &= \frac{-1}{64}\left(  96\psi(1,1/2) -192\gamma\psi(0,1/2) -96\psi(0,1/2)\psi(0,1/2) +8\psi(2,1/2) -16\gamma\psi(1,1/2)  \right. \nonumber \\ &\left. \hspace{1em}
-16\psi(0,1/2)\psi(1,1/2) -96\psi(1,3/4) +192\gamma\psi(0,3/4) +96\psi(0,3/4)\psi(0,3/4)  \right. \nonumber \\ &\left. \hspace{1em}
+16\psi(2,3/4) -32\gamma\psi(1,3/4) -32\psi(0,3/4)\psi(1,3/4) -1\psi(3,3/4) +2\gamma\psi(2,3/4)  \right. \nonumber \\ &\left. \hspace{1em}
+2\psi(0,3/4)\psi(2,3/4) +2\psi(1,3/4)\psi(1,3/4)\right) \\
\sum_{k=1}^\infty \frac{H_k}{(2k+1)^3(3k+1)} &= \frac{1}{32}\left( -144\psi(1,1/2) +288\gamma\psi(0,1/2) +144\psi(0,1/2)\psi(0,1/2) +24\psi(2,1/2) -48\gamma\psi(1,1/2)  \right. \nonumber \\ &\left. \hspace{1em}
-48\psi(0,1/2)\psi(1,1/2) -1\psi(3,1/2) +4\gamma\psi(2,1/2) +4\psi(0,1/2)\psi(2,1/2)  \right. \nonumber \\ &\left. \hspace{1em}
+144\psi(1,1/3) -288\gamma\psi(0,1/3) -144\psi(0,1/3)\psi(0,1/3)\right) \\
\sum_{k=1}^\infty \frac{H_k}{(2k+1)^3(3k+1)^2} &= \frac{1}{16}\left(  432\psi(1,1/2) -864\gamma\psi(0,1/2) -432\psi(0,1/2)\psi(0,1/2) -48\psi(2,1/2)  \right. \nonumber \\ &\left. \hspace{1em}
+96\gamma\psi(1,1/2) +96\psi(0,1/2)\psi(1,1/2) +1\psi(3,1/2) -4\gamma\psi(2,1/2)  \right. \nonumber \\ &\left. \hspace{1em}
-4\psi(0,1/2)\psi(2,1/2) -432\psi(1,1/3) +864\gamma\psi(0,1/3) +432\psi(0,1/3)\psi(0,1/3)  \right. \nonumber \\ &\left. \hspace{1em}
-24\psi(2,1/3) +48\gamma\psi(1,1/3) +48\psi(0,1/3)\psi(1,1/3)\right) \\
\sum_{k=1}^\infty \frac{H_k}{(2k+1)^3(3k+1)^3} &= \frac{1}{8}\left( -864\psi(1,1/2) +1728\gamma\psi(0,1/2) +864\psi(0,1/2)\psi(0,1/2) +72\psi(2,1/2)  \right. \nonumber \\ &\left. \hspace{1em}
-144\gamma\psi(1,1/2) -144\psi(0,1/2)\psi(1,1/2) -1\psi(3,1/2) +4\gamma\psi(2,1/2)  \right. \nonumber \\ &\left. \hspace{1em}
+4\psi(0,1/2)\psi(2,1/2) +864\psi(1,1/3) -1728\gamma\psi(0,1/3) -864\psi(0,1/3)\psi(0,1/3)  \right. \nonumber \\ &\left. \hspace{1em}
+72\psi(2,1/3) -144\gamma\psi(1,1/3) -144\psi(0,1/3)\psi(1,1/3) +2\psi(3,1/3) -4\gamma\psi(2,1/3)  \right. \nonumber \\ &\left. \hspace{1em}
-4\psi(0,1/3)\psi(2,1/3) -4\psi(1,1/3)\psi(1,1/3)\right)
\end{align}
 
\begin{align}
\sum_{k=1}^\infty \frac{H_k}{(2k+1)^3(3k+2)} &= \frac{1}{32}\left(  144\psi(1,1/2) -288\gamma\psi(0,1/2) -144\psi(0,1/2)\psi(0,1/2) +24\psi(2,1/2) -48\gamma\psi(1,1/2)  \right. \nonumber \\ &\left. \hspace{1em}
-48\psi(0,1/2)\psi(1,1/2) +1\psi(3,1/2) -4\gamma\psi(2,1/2) -4\psi(0,1/2)\psi(2,1/2)  \right. \nonumber \\ &\left. \hspace{1em}
-144\psi(1,2/3) +288\gamma\psi(0,2/3) +144\psi(0,2/3)\psi(0,2/3)\right) \\
\sum_{k=1}^\infty \frac{H_k}{(2k+1)^3(3k+2)^2} &= \frac{-1}{16}\left( -432\psi(1,1/2) +864\gamma\psi(0,1/2) +432\psi(0,1/2)\psi(0,1/2) -48\psi(2,1/2)  \right. \nonumber \\ &\left. \hspace{1em}
+96\gamma\psi(1,1/2) +96\psi(0,1/2)\psi(1,1/2) -1\psi(3,1/2) +4\gamma\psi(2,1/2)  \right. \nonumber \\ &\left. \hspace{1em}
+4\psi(0,1/2)\psi(2,1/2) +432\psi(1,2/3) -864\gamma\psi(0,2/3) -432\psi(0,2/3)\psi(0,2/3)  \right. \nonumber \\ &\left. \hspace{1em}
-24\psi(2,2/3) +48\gamma\psi(1,2/3) +48\psi(0,2/3)\psi(1,2/3)\right) \\
\sum_{k=1}^\infty \frac{H_k}{(2k+1)^3(3k+2)^3} &= \frac{1}{8}\left(  864\psi(1,1/2) -1728\gamma\psi(0,1/2) -864\psi(0,1/2)\psi(0,1/2) +72\psi(2,1/2)  \right. \nonumber \\ &\left. \hspace{1em}
-144\gamma\psi(1,1/2) -144\psi(0,1/2)\psi(1,1/2) +1\psi(3,1/2) -4\gamma\psi(2,1/2)  \right. \nonumber \\ &\left. \hspace{1em}
-4\psi(0,1/2)\psi(2,1/2) -864\psi(1,2/3) +1728\gamma\psi(0,2/3) +864\psi(0,2/3)\psi(0,2/3)  \right. \nonumber \\ &\left. \hspace{1em}
+72\psi(2,2/3) -144\gamma\psi(1,2/3) -144\psi(0,2/3)\psi(1,2/3) -2\psi(3,2/3) +4\gamma\psi(2,2/3)  \right. \nonumber \\ &\left. \hspace{1em}
+4\psi(0,2/3)\psi(2,2/3) +4\psi(1,2/3)\psi(1,2/3)\right) \\
\sum_{k=1}^\infty \frac{H_k}{(2k+1)^3(4k+1)} &= \frac{-1}{64}\left(  64\psi(1,1/2) -128\gamma\psi(0,1/2) -64\psi(0,1/2)\psi(0,1/2) -16\psi(2,1/2) +32\gamma\psi(1,1/2)  \right. \nonumber \\ &\left. \hspace{1em}
+32\psi(0,1/2)\psi(1,1/2) +1\psi(3,1/2) -4\gamma\psi(2,1/2) -4\psi(0,1/2)\psi(2,1/2)  \right. \nonumber \\ &\left. \hspace{1em}
-64\psi(1,1/4) +128\gamma\psi(0,1/4) +64\psi(0,1/4)\psi(0,1/4)\right) \\
\sum_{k=1}^\infty \frac{H_k}{(2k+1)^3(4k+1)^2} &= \frac{1}{64}\left(  192\psi(1,1/2) -384\gamma\psi(0,1/2) -192\psi(0,1/2)\psi(0,1/2) -32\psi(2,1/2)  \right. \nonumber \\ &\left. \hspace{1em}
+64\gamma\psi(1,1/2) +64\psi(0,1/2)\psi(1,1/2) +1\psi(3,1/2) -4\gamma\psi(2,1/2)  \right. \nonumber \\ &\left. \hspace{1em}
-4\psi(0,1/2)\psi(2,1/2) -192\psi(1,1/4) +384\gamma\psi(0,1/4) +192\psi(0,1/4)\psi(0,1/4)  \right. \nonumber \\ &\left. \hspace{1em}
-16\psi(2,1/4) +32\gamma\psi(1,1/4) +32\psi(0,1/4)\psi(1,1/4)\right) \\
\sum_{k=1}^\infty \frac{H_k}{(2k+1)^3(4k+1)^3} &= \frac{-1}{64}\left(  384\psi(1,1/2) -768\gamma\psi(0,1/2) -384\psi(0,1/2)\psi(0,1/2) -48\psi(2,1/2)  \right. \nonumber \\ &\left. \hspace{1em}
+96\gamma\psi(1,1/2) +96\psi(0,1/2)\psi(1,1/2) +1\psi(3,1/2) -4\gamma\psi(2,1/2)  \right. \nonumber \\ &\left. \hspace{1em}
-4\psi(0,1/2)\psi(2,1/2) -384\psi(1,1/4) +768\gamma\psi(0,1/4) +384\psi(0,1/4)\psi(0,1/4)  \right. \nonumber \\ &\left. \hspace{1em}
-48\psi(2,1/4) +96\gamma\psi(1,1/4) +96\psi(0,1/4)\psi(1,1/4) -2\psi(3,1/4) +4\gamma\psi(2,1/4)  \right. \nonumber \\ &\left. \hspace{1em}
+4\psi(0,1/4)\psi(2,1/4) +4\psi(1,1/4)\psi(1,1/4)\right) \\
\sum_{k=1}^\infty \frac{H_k}{(2k+1)^3(4k+3)} &= \frac{-1}{64}\left( -64\psi(1,1/2) +128\gamma\psi(0,1/2) +64\psi(0,1/2)\psi(0,1/2) -16\psi(2,1/2) +32\gamma\psi(1,1/2)  \right. \nonumber \\ &\left. \hspace{1em}
+32\psi(0,1/2)\psi(1,1/2) -1\psi(3,1/2) +4\gamma\psi(2,1/2) +4\psi(0,1/2)\psi(2,1/2)  \right. \nonumber \\ &\left. \hspace{1em}
+64\psi(1,3/4) -128\gamma\psi(0,3/4) -64\psi(0,3/4)\psi(0,3/4)\right)
\end{align}
 
\begin{align}
\sum_{k=1}^\infty \frac{H_k}{(2k+1)^3(4k+3)^2} &= \frac{-1}{64}\left( -192\psi(1,1/2) +384\gamma\psi(0,1/2) +192\psi(0,1/2)\psi(0,1/2) -32\psi(2,1/2)  \right. \nonumber \\ &\left. \hspace{1em}
+64\gamma\psi(1,1/2) +64\psi(0,1/2)\psi(1,1/2) -1\psi(3,1/2) +4\gamma\psi(2,1/2)  \right. \nonumber \\ &\left. \hspace{1em}
+4\psi(0,1/2)\psi(2,1/2) +192\psi(1,3/4) -384\gamma\psi(0,3/4) -192\psi(0,3/4)\psi(0,3/4)  \right. \nonumber \\ &\left. \hspace{1em}
-16\psi(2,3/4) +32\gamma\psi(1,3/4) +32\psi(0,3/4)\psi(1,3/4)\right) \\
\sum_{k=1}^\infty \frac{H_k}{(2k+1)^3(4k+3)^3} &= \frac{1}{64}\left(  384\psi(1,1/2) -768\gamma\psi(0,1/2) -384\psi(0,1/2)\psi(0,1/2) +48\psi(2,1/2)  \right. \nonumber \\ &\left. \hspace{1em}
-96\gamma\psi(1,1/2) -96\psi(0,1/2)\psi(1,1/2) +1\psi(3,1/2) -4\gamma\psi(2,1/2)  \right. \nonumber \\ &\left. \hspace{1em}
-4\psi(0,1/2)\psi(2,1/2) -384\psi(1,3/4) +768\gamma\psi(0,3/4) +384\psi(0,3/4)\psi(0,3/4)  \right. \nonumber \\ &\left. \hspace{1em}
+48\psi(2,3/4) -96\gamma\psi(1,3/4) -96\psi(0,3/4)\psi(1,3/4) -2\psi(3,3/4) +4\gamma\psi(2,3/4)  \right. \nonumber \\ &\left. \hspace{1em}
+4\psi(0,3/4)\psi(2,3/4) +4\psi(1,3/4)\psi(1,3/4)\right) \\
\sum_{k=1}^\infty \frac{H_k}{(3k+1)(4k+1)} &= \frac{1}{2}\left( -1\psi(1,1/3) +2\gamma\psi(0,1/3) +1\psi(0,1/3)\psi(0,1/3) +1\psi(1,1/4) -2\gamma\psi(0,1/4)  \right. \nonumber \\ &\left. \hspace{1em}
-1\psi(0,1/4)\psi(0,1/4)\right) \\
\sum_{k=1}^\infty \frac{H_k}{(3k+1)(4k+1)^2} &= \frac{-1}{8}\left( -12\psi(1,1/3) +24\gamma\psi(0,1/3) +12\psi(0,1/3)\psi(0,1/3) +12\psi(1,1/4) -24\gamma\psi(0,1/4)  \right. \nonumber \\ &\left. \hspace{1em}
-12\psi(0,1/4)\psi(0,1/4) +1\psi(2,1/4) -2\gamma\psi(1,1/4) -2\psi(0,1/4)\psi(1,1/4)\right) \\
\sum_{k=1}^\infty \frac{H_k}{(3k+1)(4k+1)^3} &= \frac{1}{64}\left( -288\psi(1,1/3) +576\gamma\psi(0,1/3) +288\psi(0,1/3)\psi(0,1/3) +288\psi(1,1/4)  \right. \nonumber \\ &\left. \hspace{1em}
-576\gamma\psi(0,1/4) -288\psi(0,1/4)\psi(0,1/4) +24\psi(2,1/4) -48\gamma\psi(1,1/4)  \right. \nonumber \\ &\left. \hspace{1em}
-48\psi(0,1/4)\psi(1,1/4) +1\psi(3,1/4) -2\gamma\psi(2,1/4) -2\psi(0,1/4)\psi(2,1/4)  \right. \nonumber \\ &\left. \hspace{1em}
-2\psi(1,1/4)\psi(1,1/4)\right) \\
\sum_{k=1}^\infty \frac{H_k}{(3k+1)(4k+3)} &= \frac{1}{10}\left(  1\psi(1,1/3) -2\gamma\psi(0,1/3) -1\psi(0,1/3)\psi(0,1/3) -1\psi(1,3/4) +2\gamma\psi(0,3/4)  \right. \nonumber \\ &\left. \hspace{1em}
+1\psi(0,3/4)\psi(0,3/4)\right) \\
\sum_{k=1}^\infty \frac{H_k}{(3k+1)(4k+3)^2} &= \frac{1}{200}\left(  12\psi(1,1/3) -24\gamma\psi(0,1/3) -12\psi(0,1/3)\psi(0,1/3) -12\psi(1,3/4) +24\gamma\psi(0,3/4)  \right. \nonumber \\ &\left. \hspace{1em}
+12\psi(0,3/4)\psi(0,3/4) +5\psi(2,3/4) -10\gamma\psi(1,3/4) -10\psi(0,3/4)\psi(1,3/4)\right) \\
\sum_{k=1}^\infty \frac{H_k}{(3k+1)(4k+3)^3} &= \frac{-1}{8000}\left( -288\psi(1,1/3) +576\gamma\psi(0,1/3) +288\psi(0,1/3)\psi(0,1/3) +288\psi(1,3/4)  \right. \nonumber \\ &\left. \hspace{1em}
-576\gamma\psi(0,3/4) -288\psi(0,3/4)\psi(0,3/4) -120\psi(2,3/4) +240\gamma\psi(1,3/4)  \right. \nonumber \\ &\left. \hspace{1em}
+240\psi(0,3/4)\psi(1,3/4) +25\psi(3,3/4) -50\gamma\psi(2,3/4) -50\psi(0,3/4)\psi(2,3/4)  \right. \nonumber \\ &\left. \hspace{1em}
-50\psi(1,3/4)\psi(1,3/4)\right)
\end{align}
 
\begin{align}
\sum_{k=1}^\infty \frac{H_k}{(3k+1)^2(4k+1)} &= \frac{1}{6}\left( -12\psi(1,1/3) +24\gamma\psi(0,1/3) +12\psi(0,1/3)\psi(0,1/3) +1\psi(2,1/3) -2\gamma\psi(1,1/3)  \right. \nonumber \\ &\left. \hspace{1em}
-2\psi(0,1/3)\psi(1,1/3) +12\psi(1,1/4) -24\gamma\psi(0,1/4) -12\psi(0,1/4)\psi(0,1/4)\right) \\
\sum_{k=1}^\infty \frac{H_k}{(3k+1)^2(4k+1)^2} &= \frac{1}{2}\left(  24\psi(1,1/3) -48\gamma\psi(0,1/3) -24\psi(0,1/3)\psi(0,1/3) -1\psi(2,1/3) +2\gamma\psi(1,1/3)  \right. \nonumber \\ &\left. \hspace{1em}
+2\psi(0,1/3)\psi(1,1/3) -24\psi(1,1/4) +48\gamma\psi(0,1/4) +24\psi(0,1/4)\psi(0,1/4)  \right. \nonumber \\ &\left. \hspace{1em}
-1\psi(2,1/4) +2\gamma\psi(1,1/4) +2\psi(0,1/4)\psi(1,1/4)\right) \\
\sum_{k=1}^\infty \frac{H_k}{(3k+1)^2(4k+1)^3} &= \frac{-1}{16}\left(  864\psi(1,1/3) -1728\gamma\psi(0,1/3) -864\psi(0,1/3)\psi(0,1/3) -24\psi(2,1/3)  \right. \nonumber \\ &\left. \hspace{1em}
+48\gamma\psi(1,1/3) +48\psi(0,1/3)\psi(1,1/3) -864\psi(1,1/4) +1728\gamma\psi(0,1/4)  \right. \nonumber \\ &\left. \hspace{1em}
+864\psi(0,1/4)\psi(0,1/4) -48\psi(2,1/4) +96\gamma\psi(1,1/4) +96\psi(0,1/4)\psi(1,1/4)  \right. \nonumber \\ &\left. \hspace{1em}
-1\psi(3,1/4) +2\gamma\psi(2,1/4) +2\psi(0,1/4)\psi(2,1/4) +2\psi(1,1/4)\psi(1,1/4)\right) \\
\sum_{k=1}^\infty \frac{H_k}{(3k+1)^2(4k+3)} &= \frac{-1}{150}\left(  12\psi(1,1/3) -24\gamma\psi(0,1/3) -12\psi(0,1/3)\psi(0,1/3) +5\psi(2,1/3) -10\gamma\psi(1,1/3)  \right. \nonumber \\ &\left. \hspace{1em}
-10\psi(0,1/3)\psi(1,1/3) -12\psi(1,3/4) +24\gamma\psi(0,3/4) +12\psi(0,3/4)\psi(0,3/4)\right) \\
\sum_{k=1}^\infty \frac{H_k}{(3k+1)^2(4k+3)^2} &= \frac{-1}{250}\left(  24\psi(1,1/3) -48\gamma\psi(0,1/3) -24\psi(0,1/3)\psi(0,1/3) +5\psi(2,1/3) -10\gamma\psi(1,1/3)  \right. \nonumber \\ &\left. \hspace{1em}
-10\psi(0,1/3)\psi(1,1/3) -24\psi(1,3/4) +48\gamma\psi(0,3/4) +24\psi(0,3/4)\psi(0,3/4)  \right. \nonumber \\ &\left. \hspace{1em}
+5\psi(2,3/4) -10\gamma\psi(1,3/4) -10\psi(0,3/4)\psi(1,3/4)\right) \\
\sum_{k=1}^\infty \frac{H_k}{(3k+1)^2(4k+3)^3} &= \frac{-1}{10000}\left(  864\psi(1,1/3) -1728\gamma\psi(0,1/3) -864\psi(0,1/3)\psi(0,1/3) +120\psi(2,1/3)  \right. \nonumber \\ &\left. \hspace{1em}
-240\gamma\psi(1,1/3) -240\psi(0,1/3)\psi(1,1/3) -864\psi(1,3/4) +1728\gamma\psi(0,3/4)  \right. \nonumber \\ &\left. \hspace{1em}
+864\psi(0,3/4)\psi(0,3/4) +240\psi(2,3/4) -480\gamma\psi(1,3/4) -480\psi(0,3/4)\psi(1,3/4)  \right. \nonumber \\ &\left. \hspace{1em}
-25\psi(3,3/4) +50\gamma\psi(2,3/4) +50\psi(0,3/4)\psi(2,3/4) +50\psi(1,3/4)\psi(1,3/4)\right) \\
\sum_{k=1}^\infty \frac{H_k}{(3k+1)^3(4k+1)} &= \frac{1}{36}\left( -288\psi(1,1/3) +576\gamma\psi(0,1/3) +288\psi(0,1/3)\psi(0,1/3) +24\psi(2,1/3) -48\gamma\psi(1,1/3)  \right. \nonumber \\ &\left. \hspace{1em}
-48\psi(0,1/3)\psi(1,1/3) -1\psi(3,1/3) +2\gamma\psi(2,1/3) +2\psi(0,1/3)\psi(2,1/3)  \right. \nonumber \\ &\left. \hspace{1em}
+2\psi(1,1/3)\psi(1,1/3) +288\psi(1,1/4) -576\gamma\psi(0,1/4) -288\psi(0,1/4)\psi(0,1/4)\right) \\
\sum_{k=1}^\infty \frac{H_k}{(3k+1)^3(4k+1)^2} &= \frac{-1}{12}\left( -864\psi(1,1/3) +1728\gamma\psi(0,1/3) +864\psi(0,1/3)\psi(0,1/3) +48\psi(2,1/3)  \right. \nonumber \\ &\left. \hspace{1em}
-96\gamma\psi(1,1/3) -96\psi(0,1/3)\psi(1,1/3) -1\psi(3,1/3) +2\gamma\psi(2,1/3)  \right. \nonumber \\ &\left. \hspace{1em}
+2\psi(0,1/3)\psi(2,1/3) +2\psi(1,1/3)\psi(1,1/3) +864\psi(1,1/4) -1728\gamma\psi(0,1/4)  \right. \nonumber \\ &\left. \hspace{1em}
-864\psi(0,1/4)\psi(0,1/4) +24\psi(2,1/4) -48\gamma\psi(1,1/4) -48\psi(0,1/4)\psi(1,1/4)\right)
\end{align}
 
\begin{align}
\sum_{k=1}^\infty \frac{H_k}{(3k+1)^3(4k+1)^3} &= \frac{1}{4}\left( -1728\psi(1,1/3) +3456\gamma\psi(0,1/3) +1728\psi(0,1/3)\psi(0,1/3) +72\psi(2,1/3)  \right. \nonumber \\ &\left. \hspace{1em}
-144\gamma\psi(1,1/3) -144\psi(0,1/3)\psi(1,1/3) -1\psi(3,1/3) +2\gamma\psi(2,1/3)  \right. \nonumber \\ &\left. \hspace{1em}
+2\psi(0,1/3)\psi(2,1/3) +2\psi(1,1/3)\psi(1,1/3) +1728\psi(1,1/4) -3456\gamma\psi(0,1/4)  \right. \nonumber \\ &\left. \hspace{1em}
-1728\psi(0,1/4)\psi(0,1/4) +72\psi(2,1/4) -144\gamma\psi(1,1/4) -144\psi(0,1/4)\psi(1,1/4)  \right. \nonumber \\ &\left. \hspace{1em}
+1\psi(3,1/4) -2\gamma\psi(2,1/4) -2\psi(0,1/4)\psi(2,1/4) -2\psi(1,1/4)\psi(1,1/4)\right) \\
\sum_{k=1}^\infty \frac{H_k}{(3k+1)^3(4k+3)} &= \frac{-1}{4500}\left( -288\psi(1,1/3) +576\gamma\psi(0,1/3) +288\psi(0,1/3)\psi(0,1/3) -120\psi(2,1/3)  \right. \nonumber \\ &\left. \hspace{1em}
+240\gamma\psi(1,1/3) +240\psi(0,1/3)\psi(1,1/3) -25\psi(3,1/3) +50\gamma\psi(2,1/3)  \right. \nonumber \\ &\left. \hspace{1em}
+50\psi(0,1/3)\psi(2,1/3) +50\psi(1,1/3)\psi(1,1/3) +288\psi(1,3/4) -576\gamma\psi(0,3/4)  \right. \nonumber \\ &\left. \hspace{1em}
-288\psi(0,3/4)\psi(0,3/4)\right) \\
\sum_{k=1}^\infty \frac{H_k}{(3k+1)^3(4k+3)^2} &= \frac{-1}{7500}\left( -864\psi(1,1/3) +1728\gamma\psi(0,1/3) +864\psi(0,1/3)\psi(0,1/3) -240\psi(2,1/3)  \right. \nonumber \\ &\left. \hspace{1em}
+480\gamma\psi(1,1/3) +480\psi(0,1/3)\psi(1,1/3) -25\psi(3,1/3) +50\gamma\psi(2,1/3)  \right. \nonumber \\ &\left. \hspace{1em}
+50\psi(0,1/3)\psi(2,1/3) +50\psi(1,1/3)\psi(1,1/3) +864\psi(1,3/4) -1728\gamma\psi(0,3/4)  \right. \nonumber \\ &\left. \hspace{1em}
-864\psi(0,3/4)\psi(0,3/4) -120\psi(2,3/4) +240\gamma\psi(1,3/4)  \right. \nonumber \\ &\left. \hspace{1em}
+240\psi(0,3/4)\psi(1,3/4)\right) \\
\sum_{k=1}^\infty \frac{H_k}{(3k+1)^3(4k+3)^3} &= \frac{1}{12500}\left(  1728\psi(1,1/3) -3456\gamma\psi(0,1/3) -1728\psi(0,1/3)\psi(0,1/3) +360\psi(2,1/3)  \right. \nonumber \\ &\left. \hspace{1em}
-720\gamma\psi(1,1/3) -720\psi(0,1/3)\psi(1,1/3) +25\psi(3,1/3) -50\gamma\psi(2,1/3)  \right. \nonumber \\ &\left. \hspace{1em}
-50\psi(0,1/3)\psi(2,1/3) -50\psi(1,1/3)\psi(1,1/3) -1728\psi(1,3/4) +3456\gamma\psi(0,3/4)  \right. \nonumber \\ &\left. \hspace{1em}
+1728\psi(0,3/4)\psi(0,3/4) +360\psi(2,3/4) -720\gamma\psi(1,3/4) -720\psi(0,3/4)\psi(1,3/4)  \right. \nonumber \\ &\left. \hspace{1em}
-25\psi(3,3/4) +50\gamma\psi(2,3/4) +50\psi(0,3/4)\psi(2,3/4) +50\psi(1,3/4)\psi(1,3/4)\right) \\
\sum_{k=1}^\infty \frac{H_k}{(3k+2)(4k+1)} &= \frac{-1}{10}\left(  1\psi(1,2/3) -2\gamma\psi(0,2/3) -1\psi(0,2/3)\psi(0,2/3) -1\psi(1,1/4) +2\gamma\psi(0,1/4)  \right. \nonumber \\ &\left. \hspace{1em}
+1\psi(0,1/4)\psi(0,1/4)\right) \\
\sum_{k=1}^\infty \frac{H_k}{(3k+2)(4k+1)^2} &= \frac{1}{200}\left(  12\psi(1,2/3) -24\gamma\psi(0,2/3) -12\psi(0,2/3)\psi(0,2/3) -12\psi(1,1/4) +24\gamma\psi(0,1/4)  \right. \nonumber \\ &\left. \hspace{1em}
+12\psi(0,1/4)\psi(0,1/4) -5\psi(2,1/4) +10\gamma\psi(1,1/4) +10\psi(0,1/4)\psi(1,1/4)\right) \\
\sum_{k=1}^\infty \frac{H_k}{(3k+2)(4k+1)^3} &= \frac{1}{8000}\left( -288\psi(1,2/3) +576\gamma\psi(0,2/3) +288\psi(0,2/3)\psi(0,2/3) +288\psi(1,1/4)  \right. \nonumber \\ &\left. \hspace{1em}
-576\gamma\psi(0,1/4) -288\psi(0,1/4)\psi(0,1/4) +120\psi(2,1/4) -240\gamma\psi(1,1/4)  \right. \nonumber \\ &\left. \hspace{1em}
-240\psi(0,1/4)\psi(1,1/4) +25\psi(3,1/4) -50\gamma\psi(2,1/4) -50\psi(0,1/4)\psi(2,1/4)  \right. \nonumber \\ &\left. \hspace{1em}
-50\psi(1,1/4)\psi(1,1/4)\right)
\end{align}
 
\begin{align}
\sum_{k=1}^\infty \frac{H_k}{(3k+2)(4k+3)} &= \frac{1}{2}\left(  1\psi(1,2/3) -2\gamma\psi(0,2/3) -1\psi(0,2/3)\psi(0,2/3) -1\psi(1,3/4) +2\gamma\psi(0,3/4)  \right. \nonumber \\ &\left. \hspace{1em}
+1\psi(0,3/4)\psi(0,3/4)\right) \\
\sum_{k=1}^\infty \frac{H_k}{(3k+2)(4k+3)^2} &= \frac{1}{8}\left(  12\psi(1,2/3) -24\gamma\psi(0,2/3) -12\psi(0,2/3)\psi(0,2/3) -12\psi(1,3/4) +24\gamma\psi(0,3/4)  \right. \nonumber \\ &\left. \hspace{1em}
+12\psi(0,3/4)\psi(0,3/4) +1\psi(2,3/4) -2\gamma\psi(1,3/4) -2\psi(0,3/4)\psi(1,3/4)\right) \\
\sum_{k=1}^\infty \frac{H_k}{(3k+2)(4k+3)^3} &= \frac{1}{64}\left(  288\psi(1,2/3) -576\gamma\psi(0,2/3) -288\psi(0,2/3)\psi(0,2/3) -288\psi(1,3/4)  \right. \nonumber \\ &\left. \hspace{1em}
+576\gamma\psi(0,3/4) +288\psi(0,3/4)\psi(0,3/4) +24\psi(2,3/4) -48\gamma\psi(1,3/4)  \right. \nonumber \\ &\left. \hspace{1em}
-48\psi(0,3/4)\psi(1,3/4) -1\psi(3,3/4) +2\gamma\psi(2,3/4) +2\psi(0,3/4)\psi(2,3/4)  \right. \nonumber \\ &\left. \hspace{1em}
+2\psi(1,3/4)\psi(1,3/4)\right) \\
\sum_{k=1}^\infty \frac{H_k}{(3k+2)^2(4k+1)} &= \frac{-1}{150}\left(  12\psi(1,2/3) -24\gamma\psi(0,2/3) -12\psi(0,2/3)\psi(0,2/3) -5\psi(2,2/3) +10\gamma\psi(1,2/3)  \right. \nonumber \\ &\left. \hspace{1em}
+10\psi(0,2/3)\psi(1,2/3) -12\psi(1,1/4) +24\gamma\psi(0,1/4) +12\psi(0,1/4)\psi(0,1/4)\right) \\
\sum_{k=1}^\infty \frac{H_k}{(3k+2)^2(4k+1)^2} &= \frac{1}{250}\left(  24\psi(1,2/3) -48\gamma\psi(0,2/3) -24\psi(0,2/3)\psi(0,2/3) -5\psi(2,2/3) +10\gamma\psi(1,2/3)  \right. \nonumber \\ &\left. \hspace{1em}
+10\psi(0,2/3)\psi(1,2/3) -24\psi(1,1/4) +48\gamma\psi(0,1/4) +24\psi(0,1/4)\psi(0,1/4)  \right. \nonumber \\ &\left. \hspace{1em}
-5\psi(2,1/4) +10\gamma\psi(1,1/4) +10\psi(0,1/4)\psi(1,1/4)\right) \\
\sum_{k=1}^\infty \frac{H_k}{(3k+2)^2(4k+1)^3} &= \frac{-1}{10000}\left(  864\psi(1,2/3) -1728\gamma\psi(0,2/3) -864\psi(0,2/3)\psi(0,2/3) -120\psi(2,2/3)  \right. \nonumber \\ &\left. \hspace{1em}
+240\gamma\psi(1,2/3) +240\psi(0,2/3)\psi(1,2/3) -864\psi(1,1/4) +1728\gamma\psi(0,1/4)  \right. \nonumber \\ &\left. \hspace{1em}
+864\psi(0,1/4)\psi(0,1/4) -240\psi(2,1/4) +480\gamma\psi(1,1/4) +480\psi(0,1/4)\psi(1,1/4)  \right. \nonumber \\ &\left. \hspace{1em}
-25\psi(3,1/4) +50\gamma\psi(2,1/4) +50\psi(0,1/4)\psi(2,1/4) +50\psi(1,1/4)\psi(1,1/4)\right) \\
\sum_{k=1}^\infty \frac{H_k}{(3k+2)^2(4k+3)} &= \frac{1}{6}\left( -12\psi(1,2/3) +24\gamma\psi(0,2/3) +12\psi(0,2/3)\psi(0,2/3) -1\psi(2,2/3) +2\gamma\psi(1,2/3)  \right. \nonumber \\ &\left. \hspace{1em}
+2\psi(0,2/3)\psi(1,2/3) +12\psi(1,3/4) -24\gamma\psi(0,3/4) -12\psi(0,3/4)\psi(0,3/4)\right) \\
\sum_{k=1}^\infty \frac{H_k}{(3k+2)^2(4k+3)^2} &= \frac{-1}{2}\left(  24\psi(1,2/3) -48\gamma\psi(0,2/3) -24\psi(0,2/3)\psi(0,2/3) +1\psi(2,2/3) -2\gamma\psi(1,2/3)  \right. \nonumber \\ &\left. \hspace{1em}
-2\psi(0,2/3)\psi(1,2/3) -24\psi(1,3/4) +48\gamma\psi(0,3/4) +24\psi(0,3/4)\psi(0,3/4)  \right. \nonumber \\ &\left. \hspace{1em}
+1\psi(2,3/4) -2\gamma\psi(1,3/4) -2\psi(0,3/4)\psi(1,3/4)\right) \\
\sum_{k=1}^\infty \frac{H_k}{(3k+2)^2(4k+3)^3} &= \frac{-1}{16}\left(  864\psi(1,2/3) -1728\gamma\psi(0,2/3) -864\psi(0,2/3)\psi(0,2/3) +24\psi(2,2/3)  \right. \nonumber \\ &\left. \hspace{1em}
-48\gamma\psi(1,2/3) -48\psi(0,2/3)\psi(1,2/3) -864\psi(1,3/4) +1728\gamma\psi(0,3/4)  \right. \nonumber \\ &\left. \hspace{1em}
+864\psi(0,3/4)\psi(0,3/4) +48\psi(2,3/4) -96\gamma\psi(1,3/4) -96\psi(0,3/4)\psi(1,3/4)  \right. \nonumber \\ &\left. \hspace{1em}
-1\psi(3,3/4) +2\gamma\psi(2,3/4) +2\psi(0,3/4)\psi(2,3/4) +2\psi(1,3/4)\psi(1,3/4)\right)
\end{align}
 
\begin{align}
\sum_{k=1}^\infty \frac{H_k}{(3k+2)^3(4k+1)} &= \frac{1}{4500}\left( -288\psi(1,2/3) +576\gamma\psi(0,2/3) +288\psi(0,2/3)\psi(0,2/3) +120\psi(2,2/3)  \right. \nonumber \\ &\left. \hspace{1em}
-240\gamma\psi(1,2/3) -240\psi(0,2/3)\psi(1,2/3) -25\psi(3,2/3) +50\gamma\psi(2,2/3)  \right. \nonumber \\ &\left. \hspace{1em}
+50\psi(0,2/3)\psi(2,2/3) +50\psi(1,2/3)\psi(1,2/3) +288\psi(1,1/4) -576\gamma\psi(0,1/4)  \right. \nonumber \\ &\left. \hspace{1em}
-288\psi(0,1/4)\psi(0,1/4)\right) \\
\sum_{k=1}^\infty \frac{H_k}{(3k+2)^3(4k+1)^2} &= \frac{1}{7500}\left(  864\psi(1,2/3) -1728\gamma\psi(0,2/3) -864\psi(0,2/3)\psi(0,2/3) -240\psi(2,2/3)  \right. \nonumber \\ &\left. \hspace{1em}
+480\gamma\psi(1,2/3) +480\psi(0,2/3)\psi(1,2/3) +25\psi(3,2/3) -50\gamma\psi(2,2/3)  \right. \nonumber \\ &\left. \hspace{1em}
-50\psi(0,2/3)\psi(2,2/3) -50\psi(1,2/3)\psi(1,2/3) -864\psi(1,1/4) +1728\gamma\psi(0,1/4)  \right. \nonumber \\ &\left. \hspace{1em}
+864\psi(0,1/4)\psi(0,1/4) -120\psi(2,1/4) +240\gamma\psi(1,1/4)  \right. \nonumber \\ &\left. \hspace{1em}
+240\psi(0,1/4)\psi(1,1/4)\right) \\
\sum_{k=1}^\infty \frac{H_k}{(3k+2)^3(4k+1)^3} &= \frac{1}{12500}\left( -1728\psi(1,2/3) +3456\gamma\psi(0,2/3) +1728\psi(0,2/3)\psi(0,2/3) +360\psi(2,2/3)  \right. \nonumber \\ &\left. \hspace{1em}
-720\gamma\psi(1,2/3) -720\psi(0,2/3)\psi(1,2/3) -25\psi(3,2/3) +50\gamma\psi(2,2/3)  \right. \nonumber \\ &\left. \hspace{1em}
+50\psi(0,2/3)\psi(2,2/3) +50\psi(1,2/3)\psi(1,2/3) +1728\psi(1,1/4) -3456\gamma\psi(0,1/4)  \right. \nonumber \\ &\left. \hspace{1em}
-1728\psi(0,1/4)\psi(0,1/4) +360\psi(2,1/4) -720\gamma\psi(1,1/4) -720\psi(0,1/4)\psi(1,1/4)  \right. \nonumber \\ &\left. \hspace{1em}
+25\psi(3,1/4) -50\gamma\psi(2,1/4) -50\psi(0,1/4)\psi(2,1/4) -50\psi(1,1/4)\psi(1,1/4)\right) \\
\sum_{k=1}^\infty \frac{H_k}{(3k+2)^3(4k+3)} &= \frac{1}{36}\left(  288\psi(1,2/3) -576\gamma\psi(0,2/3) -288\psi(0,2/3)\psi(0,2/3) +24\psi(2,2/3) -48\gamma\psi(1,2/3)  \right. \nonumber \\ &\left. \hspace{1em}
-48\psi(0,2/3)\psi(1,2/3) +1\psi(3,2/3) -2\gamma\psi(2,2/3) -2\psi(0,2/3)\psi(2,2/3)  \right. \nonumber \\ &\left. \hspace{1em}
-2\psi(1,2/3)\psi(1,2/3) -288\psi(1,3/4) +576\gamma\psi(0,3/4) +288\psi(0,3/4)\psi(0,3/4)\right) \\
\sum_{k=1}^\infty \frac{H_k}{(3k+2)^3(4k+3)^2} &= \frac{1}{12}\left(  864\psi(1,2/3) -1728\gamma\psi(0,2/3) -864\psi(0,2/3)\psi(0,2/3) +48\psi(2,2/3)  \right. \nonumber \\ &\left. \hspace{1em}
-96\gamma\psi(1,2/3) -96\psi(0,2/3)\psi(1,2/3) +1\psi(3,2/3) -2\gamma\psi(2,2/3)  \right. \nonumber \\ &\left. \hspace{1em}
-2\psi(0,2/3)\psi(2,2/3) -2\psi(1,2/3)\psi(1,2/3) -864\psi(1,3/4) +1728\gamma\psi(0,3/4)  \right. \nonumber \\ &\left. \hspace{1em}
+864\psi(0,3/4)\psi(0,3/4) +24\psi(2,3/4) -48\gamma\psi(1,3/4) -48\psi(0,3/4)\psi(1,3/4)\right) \\
\sum_{k=1}^\infty \frac{H_k}{(3k+2)^3(4k+3)^3} &= \frac{1}{4}\left(  1728\psi(1,2/3) -3456\gamma\psi(0,2/3) -1728\psi(0,2/3)\psi(0,2/3) +72\psi(2,2/3)  \right. \nonumber \\ &\left. \hspace{1em}
-144\gamma\psi(1,2/3) -144\psi(0,2/3)\psi(1,2/3) +1\psi(3,2/3) -2\gamma\psi(2,2/3)  \right. \nonumber \\ &\left. \hspace{1em}
-2\psi(0,2/3)\psi(2,2/3) -2\psi(1,2/3)\psi(1,2/3) -1728\psi(1,3/4) +3456\gamma\psi(0,3/4)  \right. \nonumber \\ &\left. \hspace{1em}
+1728\psi(0,3/4)\psi(0,3/4) +72\psi(2,3/4) -144\gamma\psi(1,3/4) -144\psi(0,3/4)\psi(1,3/4)  \right. \nonumber \\ &\left. \hspace{1em}
-1\psi(3,3/4) +2\gamma\psi(2,3/4) +2\psi(0,3/4)\psi(2,3/4) +2\psi(1,3/4)\psi(1,3/4)\right)
\end{align}

\end{document}